\documentclass[suppldata]{interact}

\usepackage{enumerate}
\usepackage{amsmath}
\usepackage{amssymb}
\usepackage{algorithm} 
\usepackage{algpseudocode} 
\usepackage{times} 
\usepackage{graphicx} 
\usepackage{subfigure}
\usepackage[group-separator={,}]{siunitx}

\newtheorem{thm}{Theorem}[section]
\newtheorem{prop}[thm]{Proposition}

\newtheorem{cor}[thm]{Corollary}
\newtheorem{asu}[thm]{Assumption}

\newcommand{\sym}{\operatorname{Sym}}
\newcommand{\total}{\operatorname{Total}}
\newcommand{\Total}{\total}
\newcommand{\N}{\mathcal{N}(0,1)}
\newcommand{\reals}{\mathbb{R}}

\author{\name{Thomas Flynn}\thanks{CONTACT T. Flynn. Email: tflynn@bnl.gov}
  \affil{Computational Science Initiative, Brookhaven National Laboratory, Upton, NY 11973.} }
\date{}

\begin{document}

\title{A persistent adjoint method
  with dynamic time-scaling
  and an application to mass action kinetics
  \thanks{This research was supported, in part, under National Science Foundation Grants CNS-0958379, CNS-0855217, ACI-1126113 and the City University of New York High Performance Computing Center at the College of Staten Island.}
}
\maketitle
\begin{abstract}
  In this article we consider an optimization problem where the objective function is evaluated at the fixed-point of a contraction mapping parameterized by a control variable, and optimization takes place over this control variable. Since the derivative of the fixed-point with respect to the parameter can usually not be evaluated exactly, one approach is to introduce an adjoint dynamical system to estimate gradients.  Using this estimation procedure, the optimization algorithm alternates between derivative estimation and an approximate gradient descent step. 
We analyze a variant of this approach involving dynamic time-scaling,  where after each parameter update the adjoint system is iterated until a convergence threshold is passed. 
We prove that, under certain conditions,  the algorithm can find approximate stationary points of the objective function. We demonstrate the approach in the settings of an inverse problem in chemical kinetics, and learning in attractor networks. 
\end{abstract}


\section{Introduction}
In this work we consider an optimization problem subject to a fixed-point constraint:
\begin{equation}\label{minimp}
  \min\limits_{w \in W} e(x) \quad\text{subject to} \quad x = f(x,w)
\end{equation}
where the function $f:X \times W \rightarrow X$ satisfies a contraction property, and $e : X \rightarrow \mathbb{R}$ is a loss function over points in $X$.  Throughout, $W$ is some Euclidean space and $X$ is a closed convex subset of Euclidean space.
We interpret $f(x,w)$ as specifying the evolution of a dynamical system
$x_{n+1} = f(x_{n},w)$ with parameter $w$. The contraction property means that for each $w$ there is a unique fixed-point $x^{*}(w)$. The problem can then be understood as finding the $w$ so that the equilibrium point $x^{*}(w)$ is optimal for the loss function $e$.

A first approach to problem \eqref{minimp} might be to use gradient descent, and one can invoke the implicit function theorem to obtain a formula for the derivative of $(e\circ x^{*})(w)$ in terms of the derivatives of $f$ and $e$ (see Equation \eqref{abc-eqn} below).
However, using this formula directly requires computing 
$x^{*}$,
which may only be available asymptotically, via a numerical method for solving the equation 
$f(x,w) = x$.
In addition, when considering a gradient descent type scheme, a sequence of such systems needs to be solved, corresponding to the successive 
$w_{n}$
generated. When the numerical solver for these systems is sufficiently well behaved, one would expect a gain in efficiency when the solver for
$f(x,w_{n+1}) = x$
at time $n+1$ is initialized at the approximate solution to 
$f(x,w_{n}) = x$
obtained at time $n$, a strategy known as warm-starting. To ensure some type of convergence, one must adapt the relative rates, or time-scales, of the parameter process $\{w_{n}\}$ and the numerical solver $\{x_{n}\}$ so that the derivative estimates obtained are accurate enough, without spending too much time in gradient estimation that optimization becomes prohibitively expensive. More generally, the auxiliary process need not only consist of the dynamic variables 
$\{x_{n}\}$
for solving $f(x,w) = 0$; there may be a variety of ways to construct an auxiliary dynamical system which helps in approximating derivatives. Based on these considerations, in this work we propose an algorithm for solving problem \eqref{minimp} which alternates between the two phases of
1) iterating a numerical solver to compute approximate gradients and
2) an approximate descent step. The algorithm uses dynamic time-scaling, in which the time spent iterating the solver is adapted based on properties of previous derivative estimates. The resulting algorithm is termed \textit{persistent adjoint method}; persistence refers to the way that derivative estimates are carried over from one step to the next, and adjoint refers to the method used to approximate the derivatives.
The general form of the algorithm is listed as Algorithm \ref{dynalgo} below. Under conditions on the functions $f$ and $e$ in problem \eqref{minimp} and the numerical constants used in the algorithm, we prove gradient convergence of the algorithm. Formally, defining $E=( e\circ x^*)$, we show  convergence of the sequence $ \left\{\frac{\partial E}{\partial w}(w_{n})\right\}_{n\geq 1}$ to zero.  This appears below in our main convergence result, Theorem \ref{adjgrad}.
\subsection{Related work}
Gradient based approaches to fixed-point optimization have been studied in the context of neural networks, as the problem arises from optimizing attractor networks, or neural networks with cycles in their connectivity graphs. In this context, the goal of solving problem \eqref{minimp} is to find the weights on connections between nodes so that a given input drives the network to a given steady state.
Several authors independently introduced generalizations of back-propagation to networks with cycles \cite{almeida,pineda87,atiya}. Shortly thereafter, Pineda  \cite{pineda88,pineda89} formulated a variant of the procedure as the simultaneous integration of three systems, corresponding to the underlying neural network, an adjoint system, and an approximate gradient flow. Similar procedures were analyzed in \cite{riaza,tseng} using results for singularly perturbed systems. To ensure gradient convergence, the relative time-scaling of these systems is crucial. The issue of time-scales was treated, in the more general setting of continuous-time contracting systems (in the sense of \cite{lohmiller}), by the present author in  \cite{flynn}. In that work, the author gave requirements on the time-scale parameters and initial conditions so that a type of convergence can be guaranteed. 

One stochastic variant of problem \eqref{minimp} is the problem of minimizing the average cost at the stationary distribution of a Markov chain. In that case, $f$ can be a transition operator, and $x^{*}$ is a stationary distribution. In \cite{younes} an analysis was performed of a two time-scale stochastic approximation algorithm which can be applied to some problems of this type. That author considers a procedure where one component of the system estimates gradients while the other performs an approximate descent step. Applied to the Boltzmann machine, this procedure is termed \textit{persistent contrastive divergence} \cite{tieleman}. Hornik and Kuan \cite{hornik1994} obtained a convergence result for a variant of problem \eqref{minimp} when $f$ represents a contracting neural network receiving stochastic inputs.  That work differs from ours, in that they consider a fixed-time scale, and use decreasing step-sizes. Secondly, Hornik and Kuan mostly focused on the asymptotic behavior of the algorithms, while the approach we follow here lends it self to non-asymptotic analysis.

The adjoint method is a well-known technique for computing sensitivities in dynamic optimization \cite{gilesintro,peter2010}.
Several authors have established that the necessary gradients can be approximated by iterating an auxiliary dynamical system that converges to the true derivatives \cite{giles,almeida}. The present work is aimed at establishing convergence of the \textit{overall optimization procedure} when such approximation methods are used.
Coupled with a warm-starting technique for solving the state and adjoint equations, in PDE-constrained optimization it is known as the one-shot method \cite{taasan}.
In brief, the one-shot method consists of simultaneously time-stepping an underlying integrator, an adjoint solver, and  a parameter update process.
The use of approximate gradients and warm-starting lead to much faster algorithms \cite{hazra,gunther} but these features also complicate the convergence analysis of optimization algorithms.
In order to prove convergence, the relative time-scales of the auxiliary system (state and adjoint integration) and the parameter process (approximate gradient descent) needs to be considered in analyzing these schemes.
Several authors have described practical implementations of one-shot methods with dynamic time-scaling \cite{gunther,jaworski}.
A convergence proof for a class of one-shot methods using adaptive time-scaling was given in \cite{hamdi2011}. The procedure considered in \cite{hamdi2011} is a quasi-Newton algorithm that requires computing derivatives of a doubly augmented Lagrangian function at each iteration, and uses a line search algorithm in the beginning stages of optimization to guarantee convergence. Using a doubly-augmented Lagrangian  leads to computation of higher order derivatives of the functions $f$ and $e$. In this article we focus on an algorithm that uses only the first order derivatives of $f$ and $e$. 
\subsection{Outline}
In Section \ref{sec:dynopt} we introduce the prototype algorithm for optimization with dynamic time-scaling. The theorem in that section concerns how to set the time-scale parameters in a two time-scale system so that one component remains near equilibrium, in a specific sense described below. 
Then, Section \ref{sec:appopt} discusses the relevance of dynamic approximation for optimization, specifically that a function
$E : W \rightarrow \mathbb{R}$
can be optimized by such a two time-scale procedure when its derivative can be computed by a contraction mapping; this contraction will define the auxiliary process discussed above.
In Section \ref{sec:optapp} we focus on dynamic optimization, that is, how to apply Algorithm \ref{dynalgo} in the specific context of optimizing the fixed-point of a dynamical system $f$. We review how to construct a useful auxiliary system, based on adjoint sensitivity analysis and show that it can be integrated with the dynamic approximation algorithm to yield an optimization procedure.
In Section \ref{sec:appl} this optimization procedure is applied to two problems. The first concerns an inverse problem for chemical reaction networks. The optimization problem is to tune the reaction rates of a  reaction network so that the network reaches a desired steady state when presented with a certain input. The second problem concerns attractor networks, a type of neural network with feed-back connections. Attractor networks specify dynamical systems that have unique fixed-points under certain conditions, and the optimization problem is to tune the weights of the network to have a given set of fixed-points, corresponding to a set of inputs.

\section{Dynamic approximation algorithm}\label{sec:dynopt}
A straightforward approach to solve the optimization problem \eqref{minimp} is to use  gradient descent:
\[ w_{n} = w_{n-1} - \epsilon \frac{\partial (e \circ x^{*}) }{\partial \, w} (w_{n-1}) \]
where $\epsilon>0$ is a small step-size. This approach is impractical in most cases, because  calculating the required derivative is usually at least as difficult as solving the equation $f(x,w) = x$ for $x$. However, in some cases one can find a dynamical system $z^{+} = T(z,w_{n-1})$, parameterized by $w$,  which helps to compute the gradient in the sense that
$\frac{\partial }{\partial w}(e\circ x^*)(w_{n-1}) = g( z^{*}(w_{n-1}), w_{n-1})$,
where $z^{*}(w_{n-1}) = T(z^{*}(w_{n-1}),w_{n-1})$, for some function $g$. This raises a number of possibilities for approximation based on iterating $T$ and using convergence and continuity properties to bound the error in the resulting derivative estimates.
First we study this type of approximation in a general setting, and we will investigate specific choices for $T$ and $g$  in Section \ref{sec:optapp}. 
The general type of procedure we analyze is Algorithm \ref{dynalgo} shown below.

\begin{algorithm}
\begin{algorithmic}[1]
  \State \textbf{input:} functions $T: Z \times W \rightarrow Z$, and $g: Z \times W \rightarrow W$,
  parameters $\delta$ and $\epsilon$,\\\hspace{3em} initial state $z_{0} \in Z$ and $w_{0} \in W$.
\State \textbf{initialization:} $c_{1} = \delta\|g(z_{0},w_{0})\|_{W}$ \label{c1-init}
\For {$n = 1,2,\hdots$}
\State $z_n^0 \gets z_{n-1}$
\State $i \gets 0$
\Repeat
\State $i \gets i+1$
\State $z_{n}^{i} = T(z_n^{i-1},w_{n-1})$
\Until $\|z_n^i  - z_n^{i-1}\|_Z \leq c_n$
\State $z_n \gets z_n^{i}$
\State $w_{n} \gets w_{n-1} - \epsilon g(z_{n},w_{n-1})$ \label{w-update}
\State $c_{n+1} \gets \delta\|g(z_{n},w_{n-1})\|_{W}$ \label{c-update}
\EndFor
\end{algorithmic}
\caption{Prototype optimization algorithm with dynamic time-scaling}
\label{dynalgo}
\end{algorithm}
In Algorithm \ref{dynalgo}, $z_{n}$ is deemed to be sufficiently close to equilibrium when the distance between successive iterates, 
$\|z_{n}^{i} - z_{n}^{i-1}\|_{Z}$,
falls below a threshold $c_{n}$. Then a parameter update step is taken.  The next threshold $c_{n+1}$ is based on the magnitude of the increment $\|g(z_{n},w_{n-1})\|_{W}$. 
Let us briefly explain the intuition behind the definition of $c_n$. The rationale behind this criteria is that, from the optimization perspective, a sufficient condition for the sequence 
$w_{n+1} = w_{n} - g(z_{n},w_{n-1})$
to have behavior similar to the iterations of
$w_{n+1} = w_{n} - g(z_{n-1}^*,w_{n-1})$
is the property
\begin{equation}\label{cond1}
  \| g(z_{n-1}^{*},w_{n-1}) - g(z_{n},w_{n-1})\|_W \leq \alpha\|g(z_{n},w_{n-1})\|_W
  \end{equation}
for a small $\alpha$,
where $z^{*}_{n-1}$ stands for  $z^{*}(w_{n-1})$. This is formalized below in Corollary \ref{apxgd}. In case $g$ has a Lipschitz constant $L$ with respect to $z$, a sufficient condition for inequality \eqref{cond1} is
\begin{equation}\label{cond2}
\|z_{n} - z^{*}_{n-1}\|_{Z} \leq \frac{\alpha}{L}\|g(z_{n},w_{n-1})\|_{W}.
\end{equation}
If the map $T$ is a $\beta$-contraction in $z$ then, by a standard property of contraction mappings, we can use the distance between successive iterates as a surrogate for the distance to equilibrium, via the inequality
\begin{equation*}
  \|z_{n}^{i} - z_{n-1}^*\|_{Z}   \leq \frac{\beta}{1-\beta}\|x_{n}^{i} - x_{n}^{i-1}\|
  \end{equation*}
Next, as a surrogate for the right side of Equation \eqref{cond2},  we approximate the quantity \( \left\| g(z_{n},w_{n-1}) \right\|_{W} \) by  
$\|g(z_{n-1},w_{n-2})\|_{W}$,
which is available easily since it is the magnitude of the increment applied at the previous iteration. This explains how we arrive at the stopping criteria on line 10 and the definition of $c_n$ on line 13.

Theorem \ref{tracking}, stated below, is the main result we prove about the algorithm.
It gives conditions on the functions $T$ and $g$, the constants  $\delta, \epsilon$, and the starting point $z_{0}$, so that inequality \eqref{cond2} will be satisfied for all $n\geq 1$, for any desired value of $\alpha \in (0,1)$.

\begin{thm}\label{tracking} Assume  the inputs to Algorithm 1 satisfy the following:
\begin{enumerate}
\item \label{tracking-t}
  $T:Z\times W \rightarrow Z$ is a $\beta$-contraction in $z$ and $L_{w}T$-Lipschitz in $w$:
  \begin{align*}
    \forall\, z_1,z_2 \in Z, w\in W, \,\quad  \|T(z_1, w) - T(z_2, w)\|_Z &\leq \beta \|z_1 - z_2\|_{Z}, \\
    \forall\, z\in Z, w_1,w_2 \in W, \, \quad \|T(z, w_1) - T(z, w_2)\|_Z &\leq (L_{w}T) \|w_1 - w_2\|_{W}.
    \end{align*}
\item \label{tracking-g} $g:Z\times W \rightarrow W$ is Lipschitz in $z$ and $w$, with constants $L_{z}g$ and $L_{w}g$ respectively:
  \begin{align*}
    \forall\, z_1, z_2 \in Z, w\in W,\,\quad \|g(z_1, w) - g(z_2, w)\|_W &\leq (L_{z}g) \|z_1 - z_2\|_{Z}, \\
    \forall\, z \in Z, w_1,w_2 \in W, \,\quad\|g(z, w_1) - g(z, w_2)\|_W &\leq (L_{w}g) \|w_1 - w_2\|_{W}.
    \end{align*}
\item \label{tracking-coeffs} There are numbers
  $\alpha_{\epsilon},\alpha_{c}, \alpha_{\delta}$  in $(0,1)$ such that
  \begin{align*}
    c &= 
  \alpha_{c}\frac{1}{(L_{z}g)}, \\
  \epsilon &=
  \alpha_{\epsilon}\frac{(1 - \alpha_{c})}{(L_{w}g) + (L_{z}g)(L_{w}T)/(1-\beta)}, \\
  \delta &=
  \alpha_{\delta}\frac{\alpha_{c}(1-\alpha_{c})(1-\alpha_{\epsilon})(1-\beta)}
        {(1 + \alpha_{c})(L_{z}g)\beta}.
        \end{align*}
\item \label{tracking-init} The initial point $z_0$ satisfies  $\|z_{0} - z^{*}_{0}\|_{Z} \leq c\|g(z_{0},w_{0})\|_{W}.$
\end{enumerate}
Then for all
$n \geq 1$
the variables $z_n$ are well-defined (that is, the inner loop of Algorithm 1 terminates) and
\begin{equation}\label{z-cond}
  \|z_{n} - z^{*}_{n-1}\|_{Z} \leq c\|g(z_{n},w_{n-1})\|_{W}.
  \end{equation}
\end{thm}
\begin{proof}
The theorem requires a Lipschitz constant for the map 
$w \mapsto z^{*}(w)$; 
so long as $T$ is a $\beta$-contraction in $z$, and $L_{w}T$-Lipschitz in $w$
then
\begin{equation}\label{fplip}
\|z^{*}(w_1) - z^{*}(w_2)\|_{Z} \leq \frac{(L_{w}T)}{1 - \beta}\|w_1-w_2\|_W.
\end{equation} 
  First we consider the inductive step. Fix an $n > 1$. We assume that 
$z_{n-1}$ is well-defined and
\begin{equation}\label{ind-hyp}
  \|z_{n-1} - z_{n-2}^{*}\|_{Z} \leq c \|g(z_{n-1},w_{n-2})\|_{W}.
\end{equation}
  We show that $z_n$ is well-defined and
Equation \eqref{z-cond} holds.

If
$c_n > 0$,
then the inner loop terminates in a finite  number of steps, since $T$ is a contraction mapping in
$z$, and
$z_n$ is well-defined.
If
$c_n = 0$,
then we claim the inner loop terminates after a single step, because in this case,
$z_{n}^0 = z_{n-1}^{*}$. To see this, note that by definition of $c_n$ on Line \ref{c-update} of Algorithm 1,  $c_n=0$ is equivalent to
$g(z_{n-1},w_{n-2}) = 0$. According to Equation \eqref{ind-hyp}, this entails that
\begin{equation}\label{entailer-1}
  z_{n-1} = z_{n-2}^{*},
\end{equation}
and by Line \ref{w-update} of Algorithm 1 there is no parameter update, so
\begin{equation}\label{entailer-2}
  w_{n-1} = w_{n-2}.
  \end{equation} Combining  \eqref{entailer-1} and \eqref{entailer-2}, we see that
$z_{n-1} = z_{n-1}^{*}$. Therefore $z_{n}^{0}= z_{n-1} = z_{n-1}^{*}$, and $z_n$ is well-defined. Next, we show that Equation \eqref{z-cond} holds.
Using a simple property of contraction mappings, and the definition of Algorithm \eqref{dynalgo}, the $z_{n}$ emitted by the inner loop satisfies
\begin{equation}\label{algoprop}
  \begin{split}
  \|z_n - z_{n-1}^{*}\|_Z
  &=  \|z_{n}^{i} - z_{n-1}^{*}\|_Z \\
  &\leq \frac{\beta}{1-\beta}\|z_{n}^{i} - z_{n}^{i-1}\|_Z \\
  &\leq \frac{\beta}{1-\beta}c_n 
  =
  \frac{\beta}{1-\beta}\delta\|g(z_{n-1},w_{n-2})\|_{W}.
  \end{split}
  \end{equation}
Next, applying the Lipschitz  properties of $g$ we have
\begin{equation}\label{gzn}
  \begin{split}
  \|g(z_{n-1},w_{n-2})\|_{W} &\leq 
    \|g(z_{n},w_{n-1})\|_{W} \\&+ 
    (L_{w}g)\|w_{n-1}-w_{n-2}\| _{W}+    
    (L_{z}g)\|z_{n}-z_{n-1}\|_{Z}.
    \end{split}
\end{equation}
Applying the triangle inequality  to 
$\|z_{n} - z_{n-1}\|_{Z}$ 
we get
\begin{equation}
  \begin{split}
\|z_{n} - z_{n-1}\|_{Z} &\leq 
  \|z_{n} - z_{n-1}^{*}\| _{Z} +
  \|z_{n-1}^{*} - z_{n-2}^{*}\|_{Z} + 
  \|z_{n-2}^{*} -z_{n-1}\|_{Z} \\
 &\leq 
  \left(\frac{\beta}{1-\beta}\delta + \epsilon\frac{(L_{w}T)}{1-\beta} + c\right)
     \|g(z_{n-1},w_{n-2})\|_{W}.\label{succz}
  \end{split}
  \end{equation}
Where in the second inequality we have used \eqref{algoprop}, the Lipschitz property of $z^{*}$, and the inductive hypothesis, in sequence.
Furthermore, by  definition of Algorithm \ref{dynalgo}, the $w_n$ satisfy
\begin{equation}\label{wn-satisfier}
  \|w_{n-1} - w_{n-2}\|_{W} = \epsilon\|g(z_{n-1},w_{n-2})\|_{W}.
  \end{equation}
Combining \eqref{gzn}, \eqref{succz} and \eqref{wn-satisfier}, we obtain
\begin{equation}\label{obtainer}
  \|g(z_{n-1},w_{n-2})\|_{W} \leq 
  \|g(z_{n},w_{n-1})\|_{W} + A\|g(z_{n-1},w_{n-2})\|_{W}
\end{equation}
where $A$ is defined as 
$$A = \left[\Big((L_{w}g) + (L_{z}g)\frac{(L_{w}T)}{1-\beta}\Big)\epsilon +
            (L_{z}g)c +
  (L_{z}g)\frac{\beta}{1-\beta}\delta\right].$$
We will use that $A<1$. To see why this is so, note that
\begin{align*}
  A &=
  \alpha_{\epsilon}(1-\alpha_c) +
  \alpha_c +
    \alpha_\delta\frac{\alpha_c(1-\alpha_c)(1-\alpha_{\epsilon})}{1+\alpha_c}
  \\  &
        =
  \alpha_c + (1-\alpha_c)\left(\alpha_{\epsilon} + (1-\alpha_{\epsilon})\frac{\alpha_c}{1+\alpha_c}\right) < 1.
\end{align*}
Combining \eqref{algoprop} and \eqref{obtainer}, then, 
\begin{align*}
\|z_{n} - z_{n-1}^{*}\|_{Z} &\leq 
  \frac{\beta}{1-\beta}
  \delta\frac{1}{1-A}\|g(z_{n},w_{n-1})\|_{W}
\end{align*}
We claim that the coefficient on the right-hand side is upper-bounded by $c$. Observe that
\begin{equation}\label{one-min-a}
  1 - A = ( 1 - \alpha_c )( 1 - \alpha_{\epsilon} )( 1 - \alpha_c/(1 + \alpha_c) )
\end{equation}
Using Equation \eqref{one-min-a} and the definition of $\delta$, then
\begin{equation}\label{b-1-min-b}
  \begin{split}
  \frac{\beta}{1-\beta}\delta \frac{1}{1-A} &=
  \alpha_{\delta}\alpha_c(1-\alpha_c)(1-\alpha_{\epsilon})\frac{1}{(1+\alpha_c)L_{z}g}\frac{1}{1-A} \\
  &=
  c\alpha_{\delta}(1-\alpha_c)(1-\alpha_{\epsilon})\frac{1}{(1+\alpha_c)}\frac{1}{(1-\alpha_c)(1-\alpha_{\epsilon})(1-\alpha_c/(1+\alpha_c))} \\
  &=
  c   \alpha_{\delta}\frac{1}{(1+\alpha_c)}\frac{1}{(1-\alpha_c/(1+\alpha_c))}  \\
  &=
  c   \alpha_{\delta}\frac{1}{1+\alpha_c -\alpha_c}  \\
  &=
  c \alpha_{\delta} < c.
  \end{split}
\end{equation}
This concludes our argument for the case $n>1$. Next we treat the base case $n=1$. We will show that $z_1$ is well-defined and $\|z_{1} - z_{0}^{*}\|_{Z} \leq c\|g(z_{1},w_{0})\|_{W}$.

If $c_1 > 0$, then the inner loop terminates in a finite number of steps, as a consequence of the contraction mapping theorem.
If
$c_1 = 0$,
then according to Line \ref{c1-init},  we have
$g(z_0,w_0) = 0$. But Assumption \ref{dynopt}.\ref{init-point-asu} on $z_0$ would imply $z_0 = z_0^*$, which means  $z_1^0$ is already a fixed-point of $T(\cdot,w_{0})$, and the inner loop terminates in one iteration. Therefore $z_1$ is well-defined.

Based on the assumption on $z_{0}$, and by definition of Algorithm \ref{dynalgo}, $z_1$ satisfies
\begin{equation}\label{z1prop}
  \begin{split}
    \|z_{1} - z_{0}^{*}\|_{Z} &=
    \|z_{1}^{i} - z_{0}^{*} \|_{Z} 
    \leq \frac{\beta}{1-\beta}\|z_{1}^{i}  - z_{1}^{i-1}\|_{Z} 
    \leq
    \frac{\beta}{1-\beta}\delta\|g(z_{0},w_{0})\|_{W} .
    \end{split}
\end{equation}
Applying the Lipschitz properties of $g$ yields
\begin{equation}\label{gzo}
  \|g(z_{0},w_{0})\|_{W}
  \leq
  \|g(z_{1},w_{0})\|_{W} + (L_{z}g)\|z_{1} - z_{0}\|_{Z}.
\end{equation}
Combining Assumption \ref{init-point-asu} and inequality \eqref{z1prop}, we have
\begin{equation}\label{z1z0}
  \begin{split}
    \|z_{1} - z_{0}\|_{Z}
    &\leq
    \|z_1 -  z_0^*\|_Z + \|z_0^* - z_0\|_Z 
    \leq \left(c + \frac{\beta}{1-\beta}\delta\right)\|g(z_{0},w_{0})\|_{W}.
  \end{split}
\end{equation}
Then combining \eqref{z1z0} with \eqref{gzo},
\begin{equation}\label{gz0-2}
  \|g(z_{0},w_{0})\|_{W} \leq 
  \|g(z_{1},w_{0})\|_{W} + 
  B\|g(z_{0},w_{0})\|_{W}
\end{equation}
where $B = (c + \frac{\beta}{1-\beta}\delta)$.
Combining \eqref{z1prop} with \eqref{z1z0},
\begin{align*}
  \|z_{1} - z_{0}^{*}\| _{Z}&\leq \frac{\beta}{1-\beta}\delta\frac{1}{1 - B}\|g(z_{1},w_{0})\|_{W} .
\end{align*}
Finally, note that $B<A$ and equation \eqref{b-1-min-b} imply
$\frac{\beta}{1-\beta}\delta\frac{1}{1-B} \leq c\,\alpha_{\delta} < c.$
\end{proof}

Under the same set of assumptions on $T,g$, one can develop of a variant of Algorithm \ref{dynalgo} with a fixed time-scaling, in which $T$ is iterated for a fixed number of times after each parameter update. A similar proof shows that this algorithm can also generate sequences $\{z_{n}\}, \{w_{n}\}$ with the desired properties. This variation of the procedure has been explored in \cite{flynn_mtns}. Although this results in a simpler algorithm, a goal of the dynamic time-scaling was to account for the fact that even in situations where the constants in the algorithm can be calculated, these are likely to be very conservative. 

\section{Applications to Optimization}\label{sec:appopt}
As discussed above, Algorithm \ref{dynalgo} can be applied to optimization when the functions $g$ and $T$ are such that $g(z_{n},w_{n-1})$ approximates the gradient of a function of $w$.
 Proposition \ref{dynopt} of this section makes this precise. The optimization results are based on a convergence theorem for approximate gradient descent, Proposition \ref{apxgd} below.
This result guarantees gradient convergence,
meaning convergence of the sequence
$\left\{\frac{\partial E}{\partial w}(w_{n})\right\}_{n\geq 1}$ to zero when the procedure is applied to a function $E: W \rightarrow \mathbb{R}$.

The following result on approximate gradient descent uses a condition on the approximate derivatives that combines an angle and magnitude condition. This type of requirement appears in other steepest descent results, for example \cite{bertsekas}. The difference between Proposition \ref{prop:gdadd} and the results of \cite{bertsekas} is that we are concerned with deterministic algorithms and constant step-sizes, as opposed to stochastic algorithms and decreasing step sizes.
\begin{prop}\label{prop:gdadd}
Let $E : W \rightarrow \mathbb{R}$
be a continuously differentiable function that is bounded from below and whose gradient is $L$-Lipschitz continuous.  Consider the sequence
$$w_{n+1} = w_{n} - \epsilon \left(\frac{\partial E}{\partial w}(w_{n}) + h_{n}\right).$$
If
$\|h_{n}\|_2 \leq \alpha\|\frac{\partial E}{\partial w}(w_{n})\|_2$
for some
$ \alpha \in [0,1)$ 
and 
$ \epsilon \in (0,1/L]$, then there is a $k$ such that
$E(w_{n+1}) \leq E(w_n) - k\|\frac{\partial E}{\partial w}(w_n)\|^2$. Consequently,
  $E(w_n)$ converges and $\frac{\partial E}{\partial w}(w_n) \rightarrow 0$.
\end{prop}
\begin{proof}
  We begin by  deriving an inequality of the form
  \begin{equation}\label{ineq-to-derive}
    E(w_{n+1}) \leq E(w_{n}) - k\left\|\frac{\partial E}{\partial w}(w_{n})\right\|_{2}^{2}
  \end{equation}
  for some $k>0$. This guarantees that the objective function decreases at each step. Using a second-order Taylor expansion together with the definition of $w_{n+1}$:
$$
E(w_{n+1})
\leq
E(w_{n})
- \epsilon(1-\frac{L}{2}\epsilon)\left\|\frac{\partial E}{\partial w}(w_n)\right\|_2^2
+ \epsilon|1-L\epsilon|\left\|\frac{\partial E}{\partial w}(w_n)\right\|_2\|h_n\|_2
+ \frac{L}{2} \epsilon^{2}\|h_n\|_2^{2}$$
Next, note that $|1-L\epsilon|  = 1-\epsilon$, and use the assumption on $h_n$ to get
\begin{equation}\label{descent-ineq}
E(w_{n+1})
\leq
E(w_n) - \epsilon (1-\alpha)\left(1-\frac{L}{2}\epsilon\left(1-\alpha\right)\right)
\left\|\frac{\partial E}{\partial w}(w_n)\right\|_2^2
\end{equation}
Hence the value of $k$ is
$
k=\epsilon (1-\alpha)\left(1-\frac{L}{2}\epsilon(1-\alpha)\right).
$
Since $\{E(w_n)\}_{n\geq 1}$ is non-increasing and bounded from below (by $E^*$), the limit of the sequence must exist.
Rearranging \eqref{descent-ineq}  and summing over $n=1,2,\hdots,m$,
\begin{equation}
k
\sum\limits_{n=1}^{m}\left\|\frac{\partial E}{\partial w}(w_n)\right\|_2^2
\leq
E(w_1) - E^{*}.
\end{equation}
Hence
$\frac{\partial E}{\partial w}(w_n) \rightarrow 0$.
\end{proof}

For the present purposes, it is more convenient to work with the following corollary of Proposition \ref{prop:gdadd}.
\begin{cor}\label{apxgd}
  Let  $E : W \rightarrow \mathbb{R}$ be a continuously differentiable function that is bounded from below and whose gradient is  $L$-Lipschitz continuous. Consider the sequence
  \begin{equation}\label{wform}
    w_{n+1} = w_{n} - \epsilon h_{n}.
  \end{equation}
If
$\|h_{n} - \frac{\partial E}{\partial w}(w_{n})\|_{2} \leq \alpha\|h_{n}\|_{2}$
for some $\alpha \in [0,1/2)$ 
and 
$\epsilon \in (0,1/L]$,
then $E(w_{n})$ converges
and $\frac{\partial E}{\partial w}(w_{n}) \rightarrow 0$.
\end{cor}
\begin{proof}
    In order to apply Proposition \ref{prop:gdadd}, it suffices to show that
  $\|h_n - \frac{\partial E}{\partial w}(w_n)\|_2 \leq r \|\frac{\partial E}{\partial w}(w_n)\|_2$,
  for some
  $r \in [0,1)$.    Under the assumption on $h_n$,
    \begin{align*}
    \left\|h_n - \frac{\partial E}{\partial w}(w_n)\right\|_2
    &\leq
    \alpha
    \left\| \left(h_n - \frac{\partial E}{\partial w}(w_n)\right) + \frac{\partial E}{\partial w}(w_n)\right\|_2\\
    &\leq
    \alpha
    \left\| h_n - \frac{\partial E}{\partial w}(w_n)\right\|_2
    +
    \alpha\left\| \frac{\partial E}{\partial w}(w_n)\right\|_2
    \end{align*}
    Rearranging terms, we obtain
    $$
    \left\|h_n - \frac{\partial E}{\partial w}(w_n)\right\|_2
    \leq
    \frac{\alpha}{1-\alpha}\left\|\frac{\partial E}{\partial w}(w_n)\right\|_2
      $$
    Hence we can take $r= \alpha/(1-\alpha)$. Note that
    $r \in [0,1) \iff \alpha \in [0,1/2)$.

  \end{proof}
Proposition \ref{apxgd}, which concerns gradient descent, and Theorem \ref{tracking}, regarding Algorithm \ref{dynalgo} can be linked by the next result, which gives conditions on a function $E : W \rightarrow \mathbb{R}$ that enable the application of dynamic approximation schemes for optimization. Essentially, we require that the derivative of the function should be computable by a contraction mapping.

\begin{prop}\label{dynopt}
Let 
$E : W \rightarrow \mathbb{R}$
 be a function which is bounded from below, and assume the inputs to Algorithm \ref{dynalgo}
are as follows:
\begin{enumerate}
\item The functions $T$ and $g$ satisfy Assumptions \ref{tracking}.\ref{tracking-t} and \ref{tracking}.\ref{tracking-g}, respectively, \label{dynopt-lip}
\item The function $g$ is such that
 $\frac{\partial E}{\partial w}(w) = g( z^{*}(w), w)$,  where
  $z^{*}(w) = T(z^{*}(w),w)$,  \label{dynopt-g}
\item $\epsilon,\delta,$ and $c$ are defined as in Theorem \ref{tracking}, and the constant $\alpha_{c}$ is chosen with the additional constraint \label{dynopt-c}
$\alpha_{c} < 1/2$, 
\item The initial point $z_0$ satisfies
  $
  \|z_{0} - z^{*}_{0}\|_{Z} \leq \left( \alpha_c / (L_{z}g) \right) \|g(z_{0},w_{0})\|_{W}.
  $ \label{init-point-asu}
\end{enumerate}
Then Algorithm \ref{dynalgo} generates a sequence $w_{n}$ such that
 $\frac{\partial E}{\partial w}(w_{n}) \rightarrow 0$ and $E(w_{n})$ converges.
\end{prop}
\begin{proof}
  Assumptions 1, 2, and 3 guarantee that Theorem \ref{tracking} may be applied. 
Next we show that Proposition \ref{apxgd} may be applied. It is evident that the update step for $w_{n}$ in Algorithm \ref{dynalgo} is of the required form (\ref{wform}),
where $h_{n} = g(z_{n},w_{n-1})$.
We establish that for all $n\geq 1$ the inequality
$\|h_n - \frac{\partial E}{\partial w}(w_n)\| \leq \alpha\|h_n\|$ holds, for some
$\alpha \in [0,1/2)$.
Using inequality \eqref{z-cond} with the Lipschitz of property of $g$,
\begin{equation}\label{h33}
  \|g(z_{n},w_{n-1})-g(z^{*}_{n-1},w_{n-1})\|_{2}
  \leq
  (L_{z}h)c\|g(z_{n},w_{n-1})\|_{2}
\end{equation}
Then, by Assumption\ref{dynopt-g} on $g$, and the definition of $c$, from (\ref{h33}) we obtain
\begin{equation}
  \left\|
  g(z_{n},w_{n-1})
  -
  \frac{\partial E}{\partial w}(w_{n-1})
  \right\|_{2} 
  \leq
  \alpha_{c}\|g(z_{n},w_{n-1})\|_{2}
\end{equation}
Since Assumption \ref{dynopt-c} requires that $\alpha_c < 1/2$, this establishes the feasibility of the directions $h_{n}$.
It remains to show that the step-size is also feasible. Specifically, we need to show that
$
\epsilon \leq  1/L
$,
 where $L$ is any Lipschitz-constant for $ w \mapsto \frac{\partial E}{\partial w}(w)$. Combining  Assumption \ref{dynopt}.\ref{dynopt-g}, with the bound \ref{fplip}, one such $L$ is given by 
\begin{equation}\label{alip}
L = (L_{z}g)\frac{(L_{w}T)}{1-\beta} + (L_{w}g)
\end{equation}
Theorem \ref{tracking} specifies that the step-sizes are defined as
\begin{equation}\label{algoe}
  \epsilon
  = 
  \alpha_{\epsilon}\frac{(1 - \alpha_{c})}
        {(L_{w}g) + (L_{z}g)(L_{w}T)/(1-\beta)}.
\end{equation}
Since $\alpha_{\epsilon}$ and $\alpha_c$ are in the interval $(0,1)$, the condition $\epsilon < 1/L$ holds.
\end{proof}

The utility of Proposition \ref{dynopt} depends on the availability of a dynamical system $T: Z \times W \rightarrow Z$ and function $g : Z \times W \rightarrow W$ that helps in calculating the derivative of the function $E$. As we discuss in the next section, in the case of problem (\ref{minimp}), such a pair can be constructed using adjoint sensitivity analysis.
\section{Persistent Adjoint Method}\label{sec:optapp}
In this section we show how the results of the previous section may be applied in the context of an optimization algorithm for problem (\ref{minimp}). Consider an objective function $e$ on the fixed point $x^{*}(w)$ of a contraction $f(x,w)$ depending on parameter $w$. 
A useful choice for the operator $T$ which can be used in this case can be derived from adjoint
sensitivity analysis. 
In this section we review that construction
and present conditions on the dynamics $f$ and an error function
$e$ which enable the application of Algorithm \ref{dynalgo} for optimizing the function $E(w) = (e \circ x^{*})(w)$. The conditions are essentially uniform contractivity of $f$, together with boundedness of the derivatives of $f$ and $e$. 

The way we introduce the adjoint system of equations is similar to \cite{gilesintro,baldi}. Given the contraction property and differentiability of $f$, the implicit function theorem says we may evaluate the derivative of $  (e \circ x^{*})(w)$ as
\begin{equation}\label{abc-eqn}
  \frac{\partial (e \circ x^{*}) }{\partial w} = ABC
  \end{equation}
where
\begin{align*}
A &= \frac{\partial e}{\partial x}(x^{*}(w)), \\
B &= \left(I - \frac{\partial f}{\partial x}(x^{*}(w),w)\right)^{-1}, \\
C &= \frac{\partial f}{\partial w}(x^{*}(w),w).
\end{align*}
This gives two choices for computing the derivative: in the forward method the product $BC$ is computed and then pre-multiplied by $A$, while in the adjoint method $AB$ is calculated and post-multiplied by $C$. In general these calculations have different costs, depending not only on the dimensions of the relevant matrices and vectors, but also on the details of computing their entries.
Notably, if $f$ is a contraction then we may construct an auxiliary contracting system 
$z_{n+1} = T(z_{n},w)$
to compute either $AB$ or $BC$. In the adjoint case such a $T$ is given in Proposition \ref{adjctr} below. This work is primarily focused on the adjoint formulation, as it tends to be more efficient in situations where there are more parameters than state variables. In the applications we have in mind, this tends to be the case since often the parameter is a $n\times n$ matrix that controls the interactions among $n$ state variables.  To show that the  adjoint system is contracting we use the following result which gives a sufficient condition for the interconnection of contractions to again be a contraction. It is inspired by a result for continuous time systems \cite{russo,sontag}

\begin{prop}\label{dischier}
  Let $T : X \times Y\rightarrow X$ and $U: X \times Y \rightarrow Y$ satisfy the following:
\begin{enumerate}
\item $T$ is a $\beta_{x}$-contraction in $X$ and $L_{y}T$-Lipschitz in $Y$,
\item $U$ is a $\beta_{y}$-contraction in $Y$ and $L_{x}U$-Lipschitz in $X$,
\item $(L_{x}U)(L_{y}T) < (1-\beta_{x})(1-\beta_{y})$. \label{lipcond}
\end{enumerate}
Then for any positive numbers $p_{1}, p_{2}$ such that
\begin{equation}\label{pconstr}
  \max\left\{\beta_{x} + \frac{p_2}{p_1}L_xU,\,
        \beta_{y} + \frac{p_1}{p_2}L_yT\right\} < 1\\
\end{equation}
the map 
$V: X \times Y \rightarrow X \times Y$
where
$V(x,y) = ( T(x,y), U(x,y) )$
is a $\beta$-contraction on the set $Z$ with metric $d_{Z}$ where
$\beta = 
   \max \{\beta_{x} + \frac{p_{2}}{p_{1}}L_{x}U,
           \beta_{y} + \frac{p_{1}}{p_{2}}L_{y}T\}$,
$Z = X \times Y$
and 
$d_{Z}\big( (x_{1},y_{1}) , (x_{2},y_{2}) \big) = 
    p_{1}d_{X}(x_{1},x_{2}) + p_{2}d_{Y}(y_{1},y_{2})$. 
\end{prop}
\begin{proof}
  Applying the Lipschitz properties several times and collecting terms yields
  \begin{align*}
    d_{Z}(V(x_{1},y_{1}),V(x_{2},y_{2}))
    &\leq 
    \left(p_{1}\beta_{x} + p_{2}L_{x}U)d_{X}(x_{1},x_{2} \right) +
    \left(p_{2}\beta_{y} + p_{1}L_{y}T)d_{Y}(y_{1},y_{2} \right) \\
    &\leq 
    \Big(\beta_{x} + \frac{p_{2}}{p_{1}}L_{x}U\Big)p_1d_{X}(x_{1},x_{2}) +
    \Big(\beta_{y} + \frac{p_{1}}{p_{2}}L_{y}T\Big)p_2d_{Y}(y_{1},y_{2})  \\
    &\leq
    \max\left\{\beta_{x} + \frac{p_{2}}{p_{1}} L_{x}U ,\,
           \beta_{y} + \frac{p_{1}}{p_{2}} L_{y}T\right\}
           d_{Z}((x_{1},y_{1}), (x_{2},y_{2}) )  \\
           &= \beta            d_{Z}( (x_{1},y_{1}), (x_{2},y_{2}) ) .
\end{align*}
\end{proof}
Note that if the first three conditions of Proposition \ref{dischier} are satisfied, it is always possible to find $p_1,p_2$ so that inequality (\ref{pconstr}) holds. 
The case that is most relevant for our purposes is a hierarchical system, which occurs when one of $L_yT$ or $L_xU$ is zero. For instance, if $L_{y}T$ is zero, contraction can be confirmed using any positive $p_1,\, p_2$ so that 
$\beta_{x} + \frac{p_2}{p_1}L_{x}U < 1$.

The function $f : X \times W \rightarrow X$, that describes the dynamical system, and the objective on the fixed-point $e : X \rightarrow \mathbb{R}$ should satisfy the following:
\begin{asu}\label{adasu}
The function $f$ is a contraction mapping on $X$, uniformly in $w$, and  $f$ and $e$ have continuous derivatives up to order 2, with their first and second derivatives bounded.
  In particular,  the following hold: for all $x_1,x_2 \in X, w_1,w_2 \in W$,
  \begin{align*}
    \left\|\frac{\partial f}{\partial x}(x_1,w_1)\right\|_{X} &\leq \beta_{x}
    <
    1, \\
    \left\|
    \frac{\partial f}{\partial x}(x_1,w_1) - \frac{\partial f}{\partial x}(x_2,w_1)
    \right\|_{X}
    &\leq
    (L_{x^{2}}f)
    \left\|x_1-x_2\right\|_X,
    \\
    \left\|\frac{\partial f}{\partial w}(w_1)\right\|_{2,X}
    &\leq (L_{w}f), \\
    \left\|\frac{\partial f}{\partial w}(x_1,w_1) - \frac{\partial f}{\partial w}(x_1,w_2)\right\|_{2,X}
    &\leq
(L_{w^2}f)\left\|w_1-w_2\right\|_2, \\
\left\|\frac{\partial f}{\partial w}(x_1,w_1) - \frac{\partial f}{\partial w}(x_2,w_1)\right\|_{2,X}
&\leq
(L_{x,w}f)\|x_1-x_2\|_2, \\
\left\| \frac{\partial e}{\partial x}(x_1)\right\|_{X^*}
&\leq (L_{x}e), \\
\left\|\frac{\partial e}{\partial x}(x_1) - \frac{\partial e}{\partial x}(x_2)\right\|_{X^*}
&\leq (L_{x^2}e) \|x_1 - x_2\|_{X}.
 \end{align*}
\end{asu}

We now show that the term $AB$ from \eqref{abc-eqn} maybe calculated as the fixed-point of a hierarchy of contractions. We  use $B_{X^{*}}(k)$ to refer to the ball in the dual norm: $B_{X^{*}}(k) = \left\{ y \in \mathbb{R}^{n} \big\vert \|y\|_{X^*} \leq k\right\}.$
\begin{prop}\label{adjctr}
Let $f,e$ satisfy Assumption \ref{adasu}. Define the set
$Z = 
   X \times 
   B_{X^{*}}(\frac{L_{x}e}{1-\beta_{x}}) 
$
   and the function $T^{Adj(f,e)} :  Z \times W \rightarrow Z$ by
   \begin{equation}\label{t-def}
     T^{Adj(f,e)}((x,y),w) =
     \left(
     f(x,w),
     \left(\frac{\partial f}{\partial x}(x,w)\right)^{T}y
     + \frac{\partial e}{\partial x}(x)
     \right)
   \end{equation}
Then there is a norm 
$\|(x,y)\|_{Z} = 
  p_1\|x\|_X + p_2\|y\|_{X^*}$
and $\beta<1$
so that
$T^{Adj(f,e)}$
is a $\beta$-contraction on the set $Z$ in the norm $\|\cdot\|_Z$; it suffices to take
$p_2 = 1$ and
$p_1 = 2\left(\frac{ (L_{x^2}f) (L_{x}e)}{(1-\beta_x)^2} + \frac{(L_{x^2}e)}{1-\beta_x}\right)$ and the contraction coefficient is then
$\beta = (\beta_{x} + 1)/2$.
\end{prop}
\begin{proof}
For a fixed $w$, denote by $T_y$ the map 
$T_y(x,y) =
     \frac{\partial f}{\partial x}(x,w)^{T}y + 
      \frac{\partial e}{\partial x}(x)$.
This map is a $\beta_{x}$-contraction in the norm $\|\cdot\|_{X^{*}}$, since  $\frac{\partial T_{y}}{\partial y} = (\frac{\partial f}{\partial x})^{T}$. In
addition, $T_{y}$ leaves the ball
$B_{X^{*}}(\frac{L_{x}e}{1-\beta_{x}})$ invariant:
If $\|y\|_{X^{*}} \leq (L_xe)/(1-\beta_x)$ then for any $x$,
\begin{align*}
  \|T_y(x,y)\|_{X^*} &=
  \left\| \frac{\partial f}{\partial x}(x,w)^{T}y + 
  \frac{\partial e}{\partial x}(x) \right\|_{X^*} \\
  &\leq
  \beta_x \|y\|_{X^*} + (L_{x}e) \\
  &\leq
    \beta_x \frac{(L_{x}e)}{1-\beta_x} + (L_{x}e)
    =
  \frac{(L_{x}e)}{1-\beta_x}.
  \end{align*}
The Lipschitz property of $T_y$ as a function of $x$ 
follows by the assumption on the 2nd derivatives of $f$ and $e$, and by the assumption that $y$ is bounded.
In particular, we have
\begin{align*}
  \|T_{y}(x_1,y) - T_y(x_2,y)\|_{X*}  &
    =
  \left\|
  \frac{\partial f}{\partial x}(x_1,w)^Ty + \frac{\partial e}{\partial x}(x_1)
  -
  \frac{\partial f}{\partial x}(x_2,w)^Ty - \frac{\partial e}{\partial x}(x_2)
    \right\|_{X^*} \\&
 \leq
  \left\|
  \frac{\partial f}{\partial x}(x_1,w)^Ty -
  \frac{\partial f}{\partial x}(x_2,w)^Ty\right\|_{X^*}\\&\quad+ \left\| \frac{\partial e}{\partial x}(x_1)
  - \frac{\partial e}{\partial x}(x_2) \right\|_{X^*} \\&
   \stackrel{\textbf{A}}{\leq}
  \left\|
  \frac{\partial f}{\partial x}(x_1,w)^T -
  \frac{\partial f}{\partial x}(x_2,w)^T\right\|_{X^*}\frac{(L_{x}e)}{1-\beta_x} \\&\quad+ (L_{x^2}e)\|x_1-x_2\|_{X} \\&
  \stackrel{\textbf{B}}{\leq}
  \left\|
  \frac{\partial f}{\partial x}(x_1,w) -
 \frac{\partial f}{\partial x}(x_2,w)\right\|_{X}\frac{(L_{x}e)}{1-\beta_x} + (L_{x^2}e)\|x_1-x_2\|_{X} \\&
\stackrel{\textbf{C}}{\leq}
  (L_{x^2}f)\|x_1-x_2\|_X\frac{(L_{x}e)}{1-\beta_x} + (L_{x^2}e)\|x_1 - x_2\|_X \\
  &= \left( \frac{ (L_{x^2}f)(L_{x}e)}{1-\beta_x} + (L_{x^2}e)\right)\|x_1-x_2\|_X.
\end{align*}
Step \textbf{A} follows by the bound we have just established for $\|y\|_{X^*}$. Step \textbf{B} follows from the fact that for any matrix $A$, we have  $\|A^T\|_{X^*} = \|A\|_X$, and step \textbf{C} follows by the Lipschitz properties of $\frac{\partial f}{\partial x}$. Then apply Proposition \ref{dischier} to establish the contraction property. For instance, the norm
$\|(x,y)\|_Z = p_1 \|x\|_{X} +  p_2\|y\|_{X^*}$ with $p_2= 1$ and
$$
p_1 = 2\left(\frac{ (L_{x^2}f) (L_{x}e)}{(1-\beta_x)^2} + \frac{(L_{x^2}e)}{1-\beta_x}\right)
$$
will suffice; the resulting contraction coefficient is then $(\beta_x+1)/2$.
\end{proof}


The following theorem is the main result about the persistent adjoint method, and establishes that Algorithm \ref{dynalgo} may be used to
find a stationary point of the overall function
$E = (e \circ x^*)$
when Assumption \ref{adasu} holds. 
\begin{thm} \label{adjgrad}
Let 
$f : X \times W \rightarrow X$
and
$e:X \rightarrow \mathbb{R}$
satisfy Assumption \ref{adasu},
and define the
norm $\|\cdot\|_Z$ as $\|(x,y)\|_Z = p\|x\|_X + \|y\|_{X^*}$ where
  $p=  2\left(\frac{ (L_{x^2}f) (L_{x}e)}{(1-\beta_x)^2} + \frac{(L_{x^2}e)}{1-\beta_x}\right)$.
Consider Algorithm \ref{dynalgo} with inputs
given by
\begin{enumerate}
\item $T$ is the function $T^{Adj(f,e)}$ defined as in \eqref{t-def},
\item $g$ is the function $g^{Adj(f,e)}$ defined as
  \begin{equation}\label{g-def}
  g((x,y),w) = 
  \left(\frac{\partial f}{\partial w}(x,w)\right)^{T}y,
  \end{equation}
\item \label{adj-c} The constants $\delta,\epsilon$, and $c$ are defined as in Assumption \ref{tracking}.\ref{tracking-coeffs} with $\alpha_c < 1/2, \alpha_\epsilon < 1, \alpha_{\delta} < 1$, and using the following values for the contraction coefficient $\beta$ and Lipschitz constants $L_{w}T, L_{z}g, L_{w}g:$
  \begin{align*}
    \beta &= (\beta_x+1)/2, \\
    L_{w}T &= p(L_{w}f) + (L_{x,w}f)\frac{(L_{x}e)}{1-\beta_x}, \\
  L_{z}g &= \max\left\{ (L_{w}f),\, \frac{(1-\beta_x)(L_{x,w}f)}{2(L_{x^2}f)}\right\}, \\
  L_{w}g &= (L_{w^2}f) \frac{(L_{x}e)}{1-\beta_x}. 
  \end{align*}
  \item The initial point $z_0$ satisfies $\|z_0-z^*\|_Z < (\alpha_c/(L_{z}g))\|g(z_0,w_0)\|_2$. \label{adj-init}
\end{enumerate}
Then Algorithm 1 generates a sequence $w_n$ such that $\frac{\partial (e \circ x^{*})}{\partial w}(w_{n}) \rightarrow  0$. 
\end{thm}
\begin{proof}
  We verify the conditions of Proposition \ref{dynopt}, starting with Assumption \ref{dynopt}.\ref{dynopt-lip}. That $\beta$ is a contraction coefficient for $T^{Adj(f,e)}$ in the norm $\|\cdot\|_Z$ follows from Proposition \ref{adjctr}. The bounds on the Lipschitz constants for $T^{Adj(f,e)}$ and $g$ can be derived from bounds on the Lipschitz constants of $f$ and $e$ as follows.  We begin with the Lipschitz constant $L_{w}T$. Let $z \in Z$ and $w_1,w_2$ be arbitrary. Then
 \begin{align*}
    \|T(z,w_1) - T(z,w_2)\|_Z &=
    p \|f(x,w_1) - f(x,w_2)\|_X 
\hspace{-.1em}    +\hspace{-.1em}
    \left\|\frac{\partial f}{\partial x}(x,w_1)^{T}y\hspace{-.1em} -\hspace{-.1em} \frac{\partial f}{\partial x}(x,w_2)^{T}y\right\|_{X^*} \\
    \leq 
   &\,p(L_{w}f)\|w_1-w_2\|_{2} +
   \left\|
   \frac{\partial f}{\partial x}(x,w_1)^{T}
   -
   \frac{\partial f}{\partial x}(x,w_2)^{T}\right\|_{X^{*},X^*}\|y\|_{X^*} \\
 &   \leq
    p(L_{w}f)\|w_1-w_2\|_{2} +
    \frac{(L_{x}e)}{1-\beta_x}  \left\|\frac{\partial f}{\partial x}(x,w_1) - \frac{\partial f}{\partial x}(x,w_2)\right\|_{X} \\
  &  =
    \left(p(L_{w}f) +  (L_{x,w}f)\frac{(L_{x}e)}{1-\beta_x}\right)\|w_1-w_2\|_2.
  \end{align*}
 Next we treat the Lipschitz constant $L_{w}g$:
  \begin{align*}
    \|g(z,w_1) - g(z,w_2)\|_{2} &=
    \left\|
    \left(\frac{\partial f}{\partial w}(x,w_1)\right)^{T}y -
    \left(\frac{\partial f}{\partial w}(x,w_2)\right)^{T}y
    \right\|_2 \\
    &\leq
      \left\|
      \left(\frac{\partial f}{\partial w}(x,w_1)\right)^{T}-
      \left(\frac{\partial f}{\partial w}(x,w_2)\right)^{T}
      \right\|_{X^*,2}
      \|y\|_{X^*} \\
    &\leq
    \left\|
    \frac{\partial f}{\partial w}(x,w_1)-
      \frac{\partial f}{\partial w}(x,w_2)
      \right\|_{2,X}
    \frac{(L_{x}e)}{1-\beta_x} \\
    &\leq (L_{w^2}f)\|w_1-w_2\|_2\frac{(L_{x}e)}{1-\beta_x}
  \end{align*}
  Finally we consider the Lipschitz constant $L_{z}g$:
  \begin{align*}
    \|g((x_1,y_1),w) - g((x_2,y_2),w)\|_2 =
    \left\|
    \frac{\partial f}{\partial w}(x_1,w)^{T}y_1
    -
    \frac{\partial f}{\partial w}(x_2,w)^{T}y_2\right\|_{2} \\
       \leq
        \left\|
        \frac{\partial f}{\partial w}(x_1,w)^{T}\left(y_1 -y_2\right)
        \right\|_2
        +
    \left\|
         \left(\frac{\partial f}{\partial w}(x_1,w)^{T}
         -
         \frac{\partial f}{\partial w}(x_2,w)^{T}\right)
    y_2
    \right\|_{2} \\
       \leq
        \left\|\frac{\partial f}{\partial w}(x_1,w)^{T}\right\|_{X^*,2}
        \|y_1 - y_2\|_2
          +
    \left\|
          \frac{\partial f}{\partial w}(x_1,w)^{T}
          -
          \frac{\partial f}{\partial w}(x_2,w)^{T}
          \right\|_{X^*,2}
    \|y_2
        \|_{X^*} \\
       \leq
     (L_{w}f)\|y_1 - y_2\|_{X^*}
    +
    (L_{x,w}f)\|x_1-x_2\|_X
      \frac{(L_{x}e)}{1-\beta_x}. 
  \end{align*}
  The last term on the right may be bounded as follows:
  \begin{align*}
         (L_{w}f)\|y_1 - y_2\|_{X^*}
    &+
    (L_{x,w}f)\|x_1-x_2\|_X
      \frac{(L_{x}e)}{1-\beta_x} 
    \\&\leq (L_{w}f)\|y_1-y_2\|_{X^*} + (L_{x,w}f)\frac{(1-\beta_x)}{2 L_{x^2}f} 2\frac{(L_{x^2}f)(L_{x}e)}{(1-\beta_x)^2}\|x_1-x_2\|_X
\\&        \leq (L_{w}f)\|y_1-y_2\|_{X^*} + (L_{x,w}f)\frac{(1-\beta_x)}{2 L_{x^2}f} p\|x_1-x_2\|_X  \\&
    \leq
    \max\left\{ (L_{w}f) \, , \, (L_{x,w}f)\frac{(1-\beta_x)}{2 L_{x^2}f} \right\}\|z_1-z_2\|_{Z} 
  \end{align*}
  
Next, we confirm Assumption \ref{dynopt}.\ref{dynopt-g}. We show that at the equilibrium $z^*$, the function $h$ evaluates to the derivative of the objective function $E = (e \circ x^{*})$.
By definition of $h$ and by the construction of $T$, at the fixed-point $z^* = (x^{*},y^{*})$ we have
\begin{align*}
h(z^*(w),w) &=  
  \left(\frac{\partial f}{\partial w}(x^{*},w)\right)^{T}y^{*} \\
&=
 \left(\frac{\partial f}{\partial w}(x^{*},w)\right)^{T} 
 \left(I - \frac{\partial f}{\partial x}(x^{*}(w),w)\right)^{-T} 
 \left(\frac{\partial e}{\partial x}(x^{*}(w))\right)^T \\
&= \frac{\partial( e \circ x^*)}{\partial w}(w)
\end{align*}
Assumption \ref{dynopt}.\ref{dynopt-c} and \ref{dynopt}.\ref{init-point-asu} follow directly from our Assumptions \ref{adjgrad}.\ref{adj-c} and \ref{adjgrad}.\ref{adj-init}, respectively.
This completes the proof.
\end{proof}
This theoretical result justifies the applications of the persistent adjoint method that we explore in the next section, where we consider an application to  a model fitting problem in chemical kinetics, and an application to learning in attractor networks.
\section{Numerical Experiments}\label{sec:appl}
In this section, we present two examples of how the persistent adjoint method can be applied. In the first example, we consider an inverse problem in chemical kinetics, building on an algorithm for finding equilibira first presented in \cite{gnacadja} and \cite{vandorp}. The task is to compute a matrix of reaction rates among chemical species that is compatible with observed reaction data. The data that is observed is the concentrations of the various species at (approximate)  equilibrium, corresponding to varying proportions of the initial concentrations. An iterative algorithm for computing the equilibrium concentrations serves as the function $f$ that is  used as the basis for computing the derivatives in the adjoint method. We present some numerical results using data generated from synthetic reaction networks. The second example concerns an attractor neural network model. An attractor network consists of a network of simple processing units. Unlike feed-forward networks, the connectivity graph of an attractor network may have cycles. Under certain conditions, iterating the dynamical rule of the network results in convergence to a unique fixed-point, and the problem we consider is to find the weights of the network that produce a desired set of fixed-points for a set of inputs to the individual units.
\subsection{Fixed-point iteration for chemical equilibira}
In this subsection we consider an application to chemical kinetics.
We restrict ourselves to a certain class of chemical networks known as heterodimerization networks. In a heterodimerization network there are $n$ \textit{simple species}, denoted $X_1,\hdots,X_n$, and for each pair of simple species there is a corresponding \textit{complex species} $X_{i,j}$, where $i\neq j$, and $X_{i,j} = X_{j,i}$.  There are two types of reactions that  occur in a heterodimerization network.
First,  pairs of simple species $X_{i}, X_{j}$
combine to form the corresponding complex species
$X_{\{i,j\}}$ at rate $e^{w_{i,j}}$, written in symbolic form as
\begin{equation}\label{chem1}
X_{i} + X_{j} \xrightarrow{\makebox[1cm]{$e^{w_{i,j}}$}} X_{\{i,j\}}.
\end{equation}
The second type of reaction is that complex species degrade at unit rate into their constituents:
\begin{equation}\label{chem2}
X_{\{i,j\}} \xrightarrow{\makebox[1cm]{$1$}} X_{i} + X_{j}.
\end{equation}
Continuous time mass-action kinetics \cite{anderson} specifies a flow on the concentration variables $\{x_{i}\}, \{x_{\{i,j\}}\}$, from the formal equations (\ref{chem1}, \ref{chem2}). These equations are as follows:
\begin{align*}
  \frac{d x_i}{d t}(t) &=  \sum\limits_{j: \{i,j\} \in D}x_{\{i,j\}}(t)
  -\sum\limits_{j:\{i,j\} \in D} e^{w_{i,j}}x_{i}(t)x_j(t), \\
  \frac{d x_{\{i,j\}}}{d t}(t) &= -x_{\{i,j\}}(t) + e^{w_{i,j}}x_{i}(t)x_{j}(t).
  \end{align*}
where $D$ is the set pairs of simple species which react with each other. These equations imply that, at equilibrium, the concentrations must satisfy
\begin{align}
  \sum\limits_{j:\{i,j\} \in D}e^{w_{i,j}}x_{i}x_{j} &= \sum\limits_{j:\{i,j\} \in D}x_{\{i,j\}},  \notag \\
  x_{\{i,j\}} &= e^{w_{i,j}}x_{i}x_{j}. \label{complex-from-simple}
  \end{align}
In particular, \eqref{complex-from-simple} implies that the equilibrium concentrations of the complex species can be computed from the equilibrium concentrations of simple species.
 Furthermore, there is a conservation law that says
$b_{i}(t) = x_{i}(t) + \sum_{j:\{i,j\} \in D}x_{\{i,j\}}(t)$ remains constant. 

 It was shown in \cite{gnacadja} (see also \cite{vandorp}) that the components of the equilibrium for the simple species may be calculated as the fixed-point of the map
$F:\mathbb{R}_{> 0}^{N} \rightarrow \mathbb{R}_{> 0}^{N}$ given by
\begin{equation}\label{big-f}
  F_{i}(x) =
   \frac{b_i}
        { 1 + \sum\limits_{j:\{i,j\} \in D}x_{j}e^{w_{i,j}} }
        \end{equation}
A rate of convergence for the fixed-point iterations can be derived as well; \cite{gnacadja} showed that the map $F$ is a contraction in the Thompson metric 
$d(u,v) = \| \log u - \log v \|_{\infty}$.
The equilibrium concentration depends on the reaction rates $w$ and the total-concentrations $b$.

For our purposes, it will be more useful to work with a function closely related to $F$, that is obtained by working with the log-space of chemical concentrations.
Formally, denote by $\operatorname{Sym}_{n}(\mathbb{R})$ the set of $n \times n$ symmetric matrices with real-valued entries. For a vector $b \in \mathbb{R}^{n}$ define
$\mathbb{R}^{n}_{\leq b} = \{x \in \mathbb{R}^{n} | x_{i} \leq b_{i}, 1 \leq i \leq n\}$.
Let $X = \mathbb{R}^{n}_{\leq b}$, $W =  \operatorname{Sym}_{n}(\mathbb{R})$ and define the function 
$f(\cdot,\cdot\,;b) : X  \times W \rightarrow X$  as follows. The $n$ component functions $f_{i} : X \times W \rightarrow \mathbb{R}_{\leq b_{i}}$ are given by
\begin{equation}\label{crnctrc}
  f_{i}(x,w;b) =  b_{i} -\log\bigg(1+\sum\limits_{j\neq i}^{n}e^{w_{i,j} + x_{j}}\bigg)
\end{equation}
The $b_{i}$ represent the (logarithm of) the total concentration for species $i$.
The next result says that the function $f$ is a contraction for all values of $b,w$. 
\begin{prop}\label{crnprop}
  For any $n\times n$ matrix $w$ and any
  $b\in \mathbb{R}^{n}$ the map $f$ defined in \eqref{crnctrc} is a contraction in the norm
  $\|\cdot\|_{\infty}$ on the set
  $\mathbb{R}^{n}_{\leq b} = \{x \in \mathbb{R}^{n} | x_{i} \leq b_{i}, 1 \leq i \leq n\}$.
  A contraction coefficient is $\beta_x = \frac{M}{1+M},$ with $M$ defined as
  $$ M = \bigg(\max_{1\leq i\leq n}\sum\limits_{j\neq i}^{n}e^{w_{i,j}}\bigg)\|e^b\|_{\infty}. $$
\end{prop}
\begin{proof}
The derivative of $f$ with respect to $x$ is
\[
\frac{\partial f_{i}}{\partial x_{j}}(x,w;b) =
\begin{cases}
  0 &\text{ if } j = i, \\
 -\dfrac{e^{w_{i,j} + x_{j}}}
{1 + \sum_{k\neq i}^{n}e^{w_{i,k} + x_{k}}} &\text{ if } j \neq i. \end{cases}
\]
The matrix norm induced by the vector norm $\|\cdot\|_{\infty}$ is the maximum-row-sum. Specifically, we have
\begin{align*}
  \left\|
  \frac{\partial f}{\partial x}(x,w;b)
  \right\|_{\infty} =
\max_{1\leq i\leq n}
  \frac{\sum_{j\neq i}^{n}e^{w_{i,j}  + x_{j}}}
       {1 + \sum_{j\neq i}^{n}e^{w_{i,j}  + x_{j}}} 
\end{align*}
By the definition of $f$, we may assume each
$x_{i} < b_{i}$. This means that the norm of the derivative can be bounded from above by
$\left\|\frac{\partial f}{\partial x}(x,w;b)\right\| \leq \frac{M}{1 + M}$
where
$$M =
\bigg(\max_{1\leq i\leq n}\sum\limits_{j\neq i}^{n}e^{w_{i,j}}\bigg)\|e^b\|_{\infty}.$$
\end{proof}

In light of this contraction result, the function $f$ is guaranteed to have a unique fixed-point for each $w$ and $b$. The fixed-point determined by a specific $w$ and $b$ is denoted $x^{*}(w;b)$.
\subsubsection{Problem formulation}\label{prob-form}
Let
$b^1,\hdots,b^m$
be a set of initial (log) concentration vectors, and let
$\widetilde{x}^1$, $\hdots$, $\widetilde{x}^m$ be the corresponding observed equilibrium concentrations. Optimization seeks to identify a rate matrix $w$ such that
$f( \widetilde{x}^i, w; b^i) =  \widetilde{x}^i$ for each $i=1,2,\hdots,m$. This can be expressed
in the form \eqref{minimp} as follows.
Define the functions
$f^{\operatorname{Total}} :X^{10} \times W \to\ X^{10}$ and $e^{\operatorname{Total}} : X^{10} \to \reals$ as :
\begin{align}
f^{\operatorname{Total}}((x^1,\hdots,x^{10}),w)
&=
\left(
f(x^{1},w;b^1),\hdots, f(x^{10},w;b^{10})
\right), \label{f-total}
\\
e^{\operatorname{Total}}(x^{1},\hdots,x^{10})
&= \frac{1}{m}\sum\limits_{i=1}^{m}\|x^i - \widetilde{x}^i\|^{2}_{2}. \label{e-total}
\end{align}
The optimization problem can then be stated as
\begin{equation}\label{equiv-problem}
  \begin{split}
  \min\limits_{w \in \sym_n} &e^{\total}(x^1,\hdots,x^{10})\\
  \text{ subject to }
  &(x^1,\hdots,x^{10}) = f^{\total}((x^1,\hdots,x^{10}),w).
  \end{split}
\end{equation}

\begin{figure}[t]
  \centering
  \begin{subfigure}{
    \includegraphics[width=0.34\linewidth]{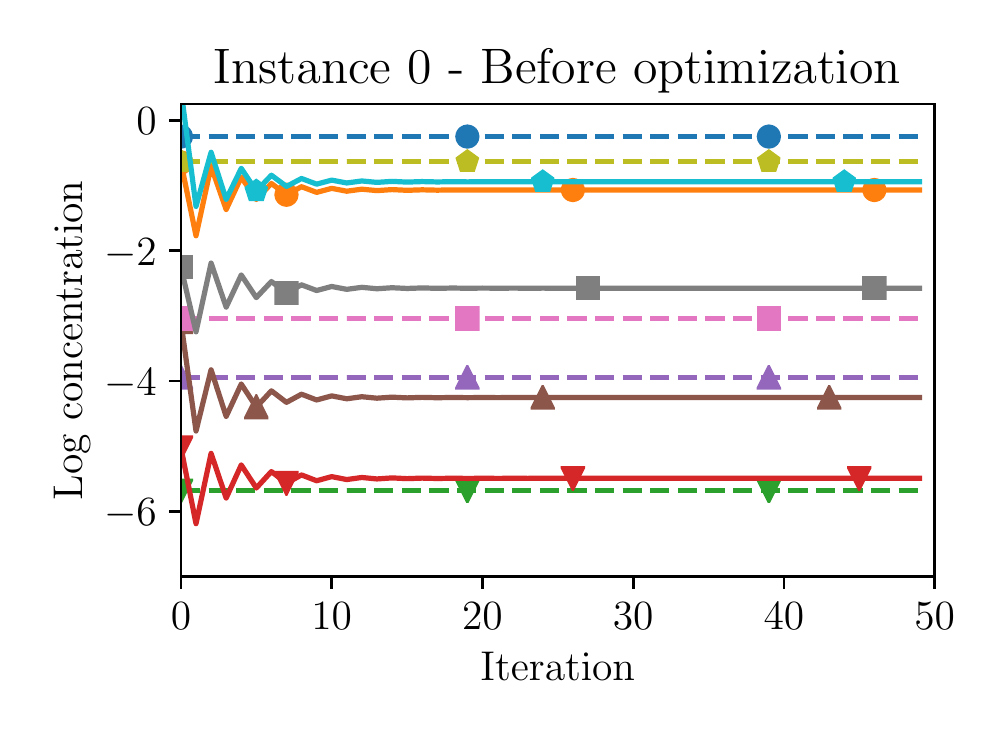}}
  \end{subfigure}
  \begin{subfigure}{\hspace{-1em}
    \includegraphics[width=0.33\linewidth]{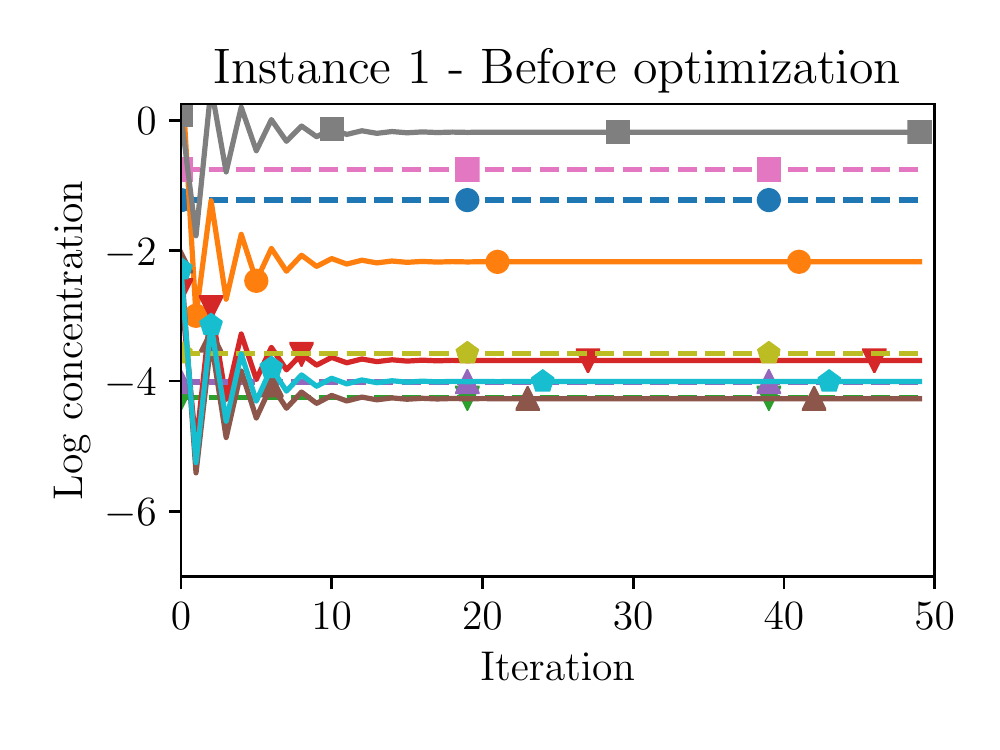}}
  \end{subfigure}
  \begin{subfigure}{\hspace{-1em}
    \includegraphics[width=0.33\linewidth]{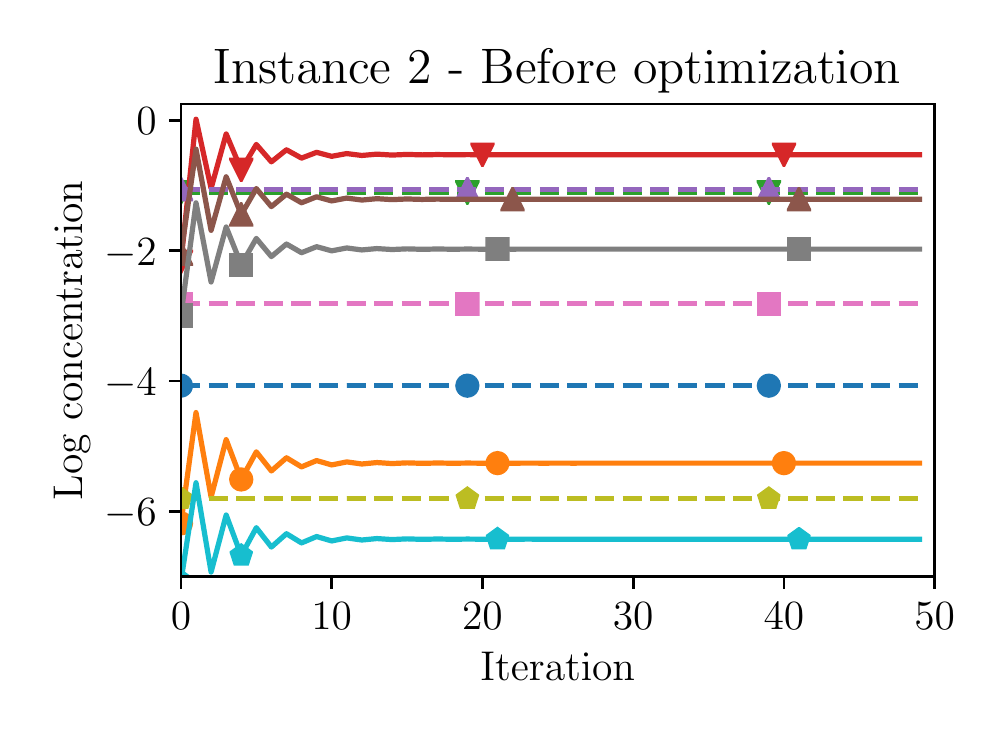}}
  \end{subfigure}  \\
  \begin{subfigure}{
    \includegraphics[width=0.34\linewidth]{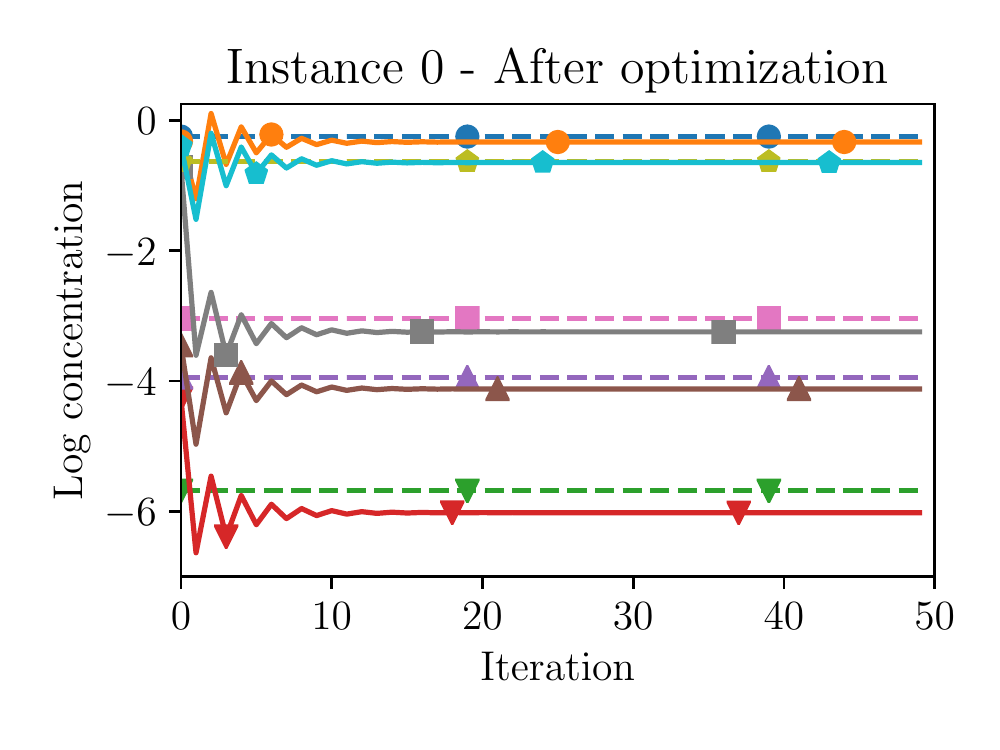}}
  \end{subfigure}
  \begin{subfigure}{\hspace{-1em}
    \includegraphics[width=0.33\linewidth]{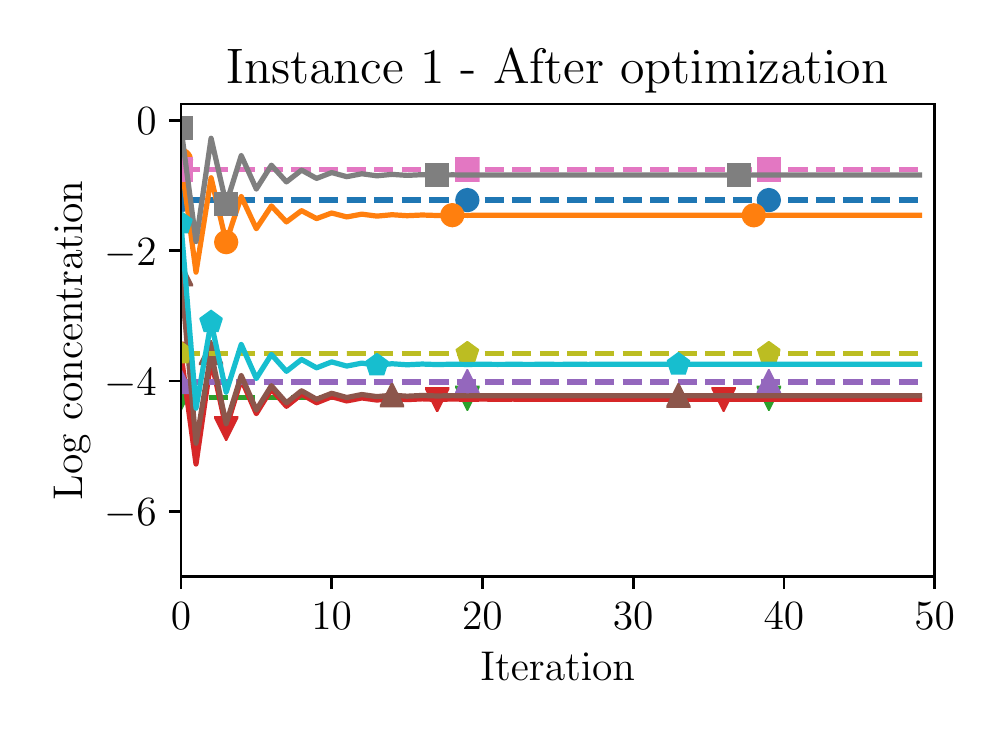}}
  \end{subfigure}
  \begin{subfigure}{\hspace{-1em}
    \includegraphics[width=0.33\linewidth]{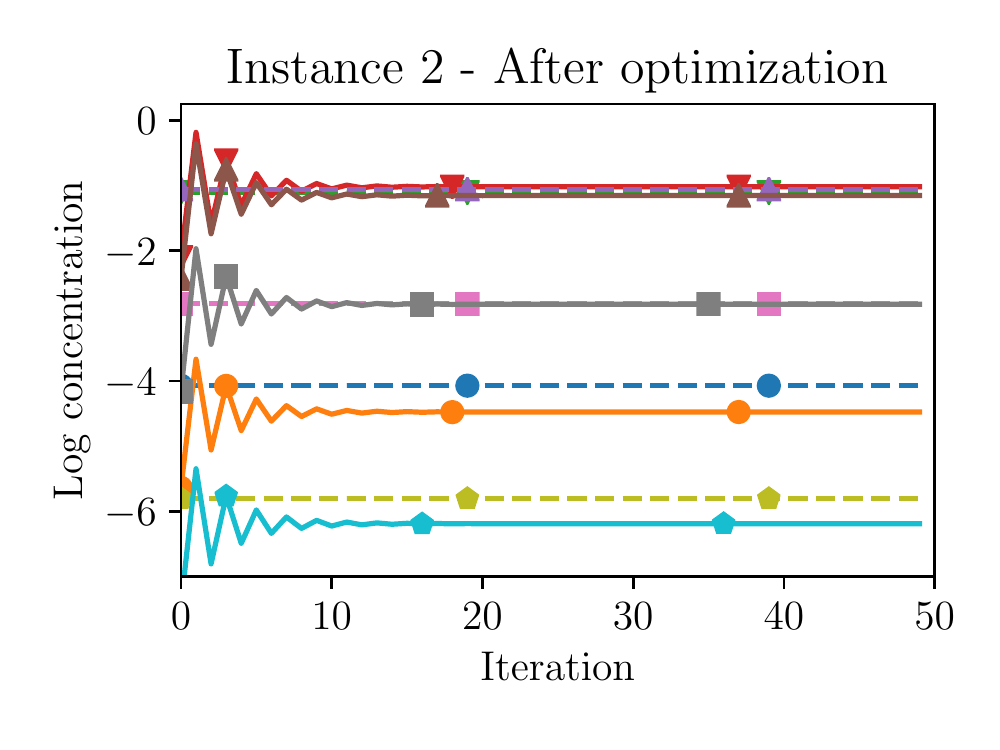}}
  \end{subfigure}  
  
  \caption{Trajectories generated in response to different inputs before (top row) and after (bottom row) optimization. The dashed lines represent the target concentrations. }\label{fig:trajectories}

\end{figure}

Note that $f^{\Total}$ can be viewed as a parallel combination 
of $m$ instances of the system \eqref{crnctrc},
each subsystem receiving different inputs in the form of total concentrations.
Parallel combinations of contractions are also contractions. Formally, a recursive application of Proposition \ref{dischier} shows that for all parameters $w$ the function $f^{\Total}$ is a contraction in the norm $\|\cdot\|_X$, defined as 
\begin{equation}\label{the-norm}
  \|(x_{1}, x_{2}, x_{4}, x_{4})\|_{X} = \sum\limits_{1\leq i \leq 4}\|x_{i}\|_{\infty}.
  \end{equation}
A contraction coefficient relative to the norm \eqref{the-norm} is
$\beta_{x}^{\Total}$, defined as   $  \beta_{x}^{\Total} = \frac{M^{\Total}}{1+M^{\Total}},$ where $M^{\Total}$ is 
\begin{equation}\label{pllctr}
 M^{\Total}
  =
  \bigg(\max_{1\leq i\leq n}\sum\limits_{j\neq i}^{n}e^{w_{i,j}}\bigg)\max_{1\leq i \leq m}
  \left\|e^{b^i}\right\|_{\infty}.
\end{equation}
Note also that in this case the dual norm is given by $\|(x_{1},x_{2},x_{3},x_{4})\|_{X^{*}} = \max_{1\leq i\leq4}\|x_{i}\|_{1}$.

Proposition \ref{crnprop} together with the results in Section \ref{sec:optapp} suggest that the persistent adjoint method can be used to find an approximate solution to the problem \eqref{equiv-problem}. Those results use the assumption that the convergence rate and various bounds on derivatives of $f$ can be bounded independently of $w$, but Proposition \ref{crnprop} suggests that the Lipschitz constant of $f$ with respect to $x$ can go to 1 as 
$\|w\| \rightarrow \infty$.
This could be addressed by performing optimization on a constrained class of models for which the various derivatives remain bounded, but for simplicity optimization was performed on the plain unconstrained model. Values for the parameters $\epsilon$ and $\delta$ which define the step-size and time-scale in Algorithm \ref{dynalgo} were determined experimentally.
\begin{figure}[t]
\begin{center}
  \begin{subfigure}{ \includegraphics[width=.5\linewidth]{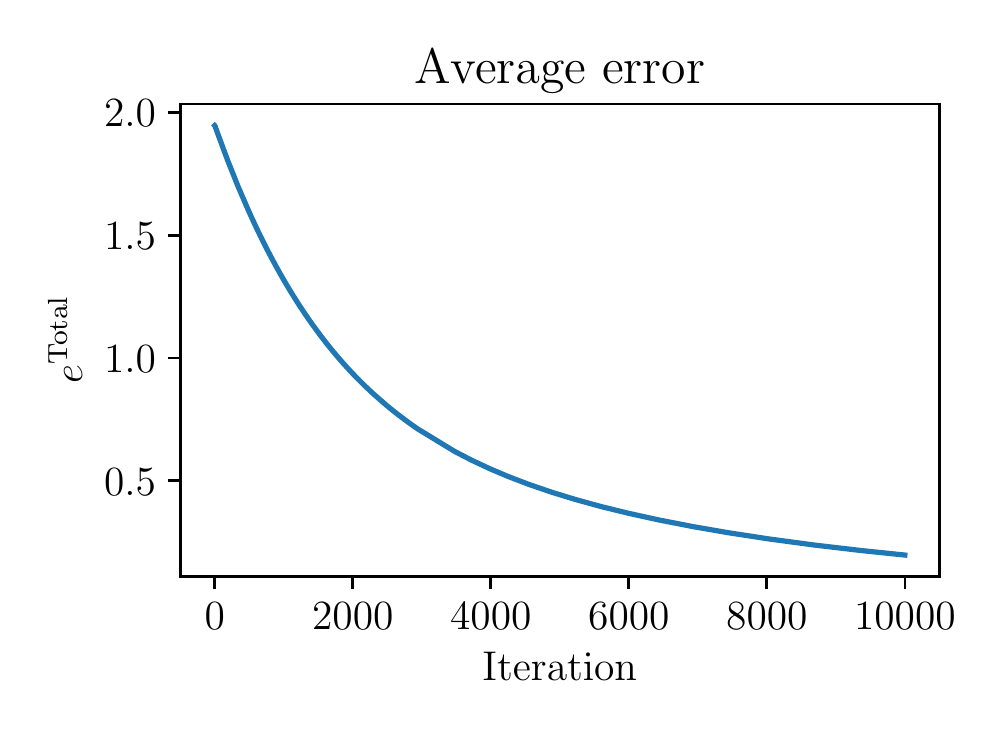}}
  \end{subfigure}
  \begin{subfigure}{\hspace{-1em}
  \includegraphics[width=.5\linewidth]{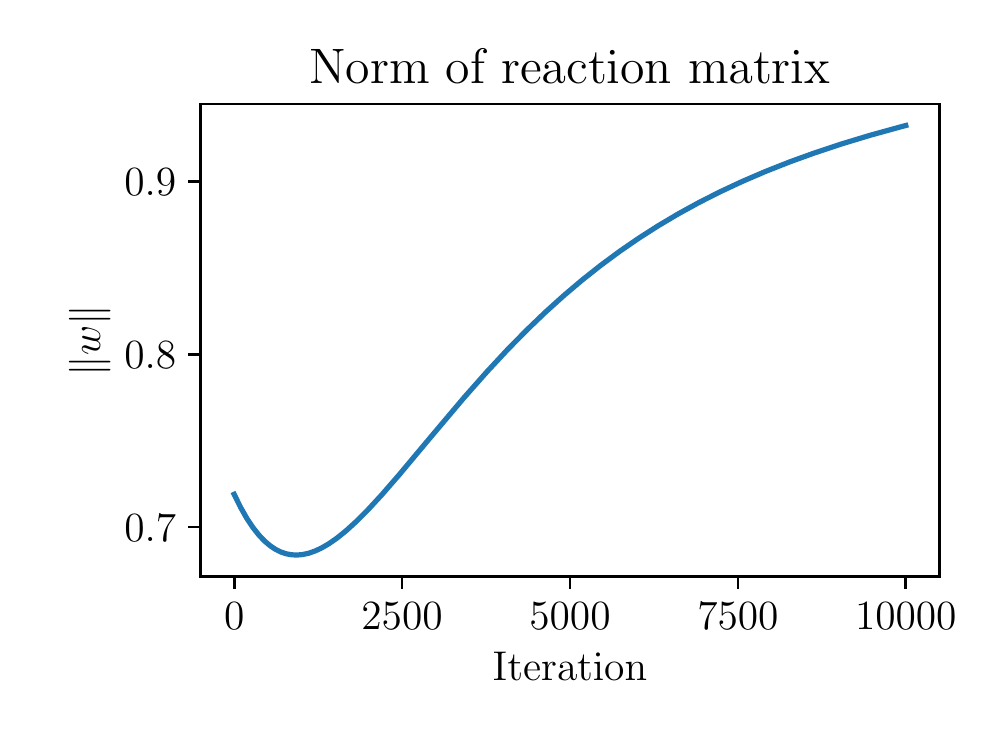}}
  \end{subfigure}
  \begin{subfigure}{    \includegraphics[width=.5\linewidth]{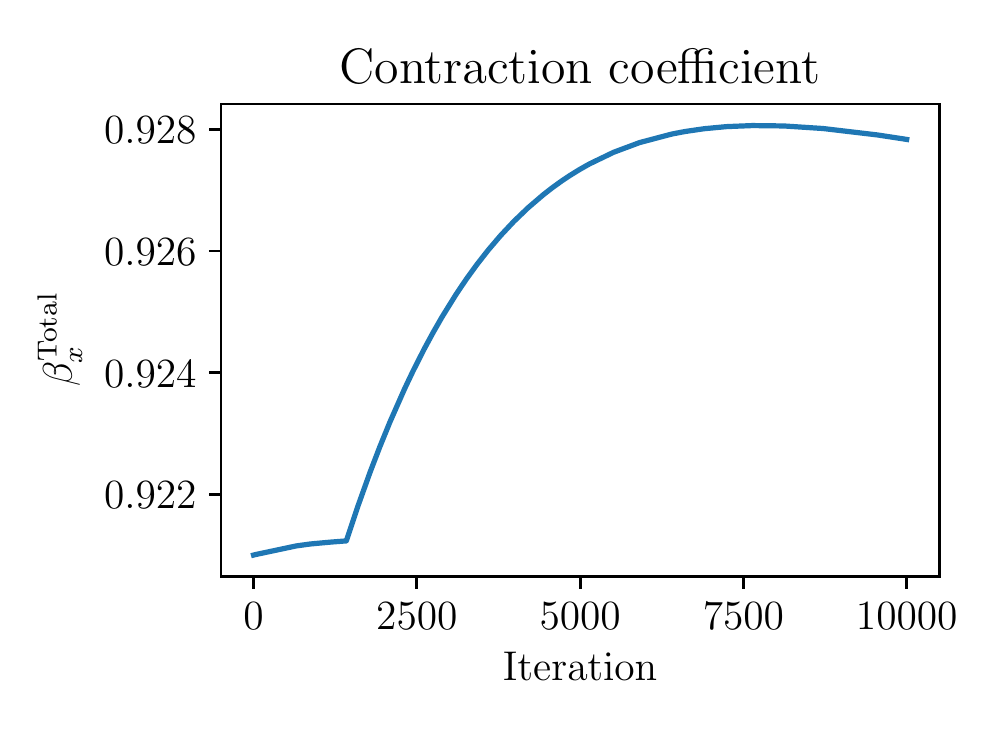}}
  \end{subfigure}
  \begin{subfigure}{\hspace{-1em}
 \includegraphics[width=.5\linewidth]{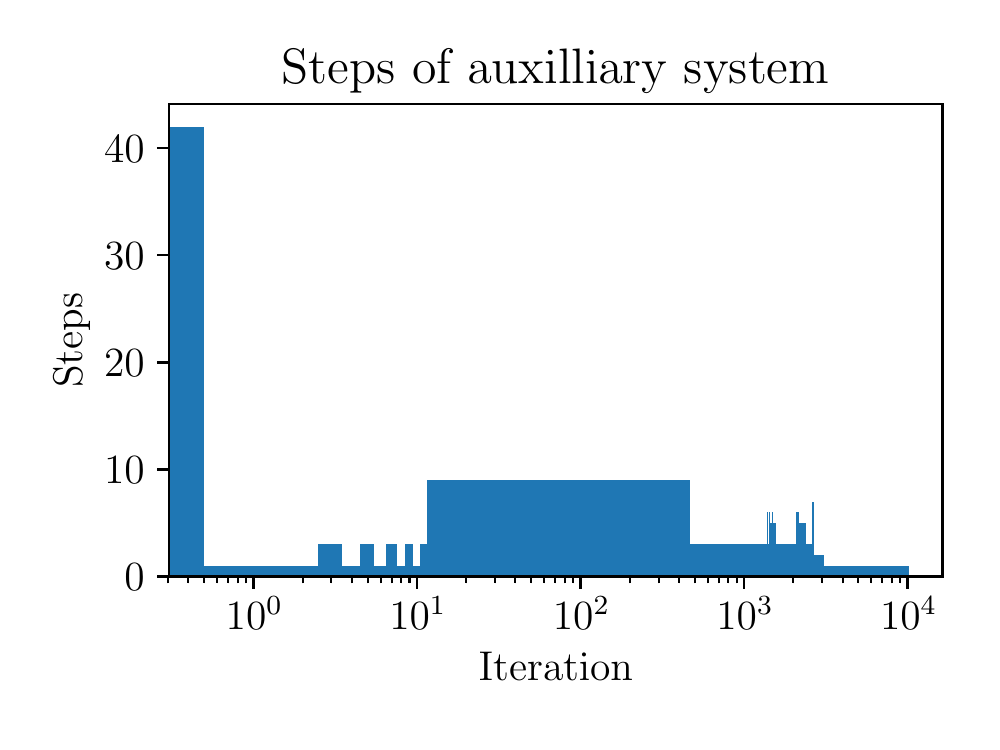}}
  \end{subfigure}
\caption{Upper left: The error over time. This is the function $e^{\total}$, defined at \eqref{e-total}, evaluated at approximate equilibrium concentrations at each iteration. Upper right: The norm of the reaction matrix after each step of optimization. Note that this is the Euclidean (sum-of-squares) norm. Lower left: The contraction coefficient of the solver for fixed-points of the chemical system as optimization progresses. Lower right: The number of auxiliary steps required between update steps. }\label{fig:epersist}
\end{center}
\end{figure}

\subsubsection{Optimization results}
Data was generated for our example problem as follows. We let $\N$ denote the normal distribution with mean zero and unit variance.
We begin with a random rate matrix $w$, whose entries are sampled from $\N$.
Then we generate 10 random total concentration vectors $b^1, \hdots, b^{10}$,  whose entries are also from $\N$.
For each concentration vector we compute (approximate) equilibrium concentration vectors $\widetilde{x}^{1} , \hdots, \widetilde{x}^{10}$. We  obtain the starting point for optimization by defining $w^{0}$ as a random matrix with entries sampled from $\N$.

Algorithm 1 was run with the following inputs:
\begin{itemize}
\item
  $
  T = T^{Adj(f^{\total},\, e^{\total})},
  g = g^{Adj(f^{\total},\, e^{\total})}$ with $f,e$  as in \eqref{f-total}, \eqref{e-total}, resp.,
\item
  $ \epsilon = 0.4, \delta = 0.01 $
 \item
   $z_{0} = ((x^1,\hdots,x^{10}),(y^1,\hdots,y^{10}))$ is set to $(0,0) \in \mathbb{R}^{100}$
 \item
   $\|\cdot\|_W$ is the Frobenius norm: $\|w\|_W = ( \sum_{i=1}^{n}\sum_{j=1}^{n}w_{i,j}^2)^{1/2}$.
   \item
$ \| \cdot \|_{Z} $ is $\|(x,y)\|_{Z} = \|x\|_{X} + \|x\|_{X^{*}}$ \quad (See  \eqref{the-norm}.)
\end{itemize}

The persistent adjoint method was run with the settings above for $\num{50000}$ iterations, and we present some results of the optimization in Figures \ref{fig:trajectories} and \ref{fig:epersist}. We begin with a qualitative picture of the results in Figure \ref{fig:trajectories}. The first row shows how three of the inputs drive the system \eqref{crnctrc} before training, when the parameters are initialized randomly. Each columns corresponds to a different input, and one can see that initially the equilibrium concentrations  do not converge to the the targets, shown in the dashed line. The second row shows how the base system responds to each input after training. We see that each input produces an equilibrium concentration that is nearly equal to the target concentration.

Figure \ref{fig:epersist} shows a quantitative picture of optimization performance, by tracking several properties of the optimization trajectory.
The upper left plot of Figure \ref{fig:epersist} shows  $E$ as training progresses. For each $n$, the error $E(w_{n})$ was estimated by first iterating $f^{\total}$ until a convergence threshold was reached, and then computing $e^{\total}(x)$ on the resulting estimate of the fixed-point of the system.  The upper right plot of Figure \ref{fig:epersist} shows how the norm of the reaction matrix evolves during the training.  During optimization, we calculated the bound on the contraction coefficient (given prior to Equation \eqref{pllctr} above), and the result is shown in Figure  \ref{fig:epersist}, lower left.  In the persistent adjoint method, the time spent in gradient estimation varies across iterations, and we track this quantity as well. This is shown in Figure \ref{fig:epersist}, lower right, where we plot the number of steps of the auxiliary system at each between gradient updates. Note that after the first few hundred iterations, the auxiliary system only needs one iteration after each parameter update. Overall, these results suggest that the persistent adjoint method could be a useful algorithm for finding reaction rates consistent with observed data.

\subsection{Attractor networks}
The second example we consider is a learning problem in attractor networks, a type of neural network with feed-back connections. The network is assumed to have $n$ nodes, and the state of each node is represented by a single number. Therefore the state space of the entire network is a vector in $X=\mathbb{R}^n$, and the parameters of the network are stored in a matrix in $W = \mathbb{R}^{n\times n}$, while external input is represented as a vector $b \in \reals^n$. At each time $n$, the state of each node is updated by a local computation involving its neighboring states. Formally, the function
$f(\cdot,\cdot;u) : X\times W \to X$ is defined using the $n$ component function $f_i :X\times W \to \reals$ as
\begin{equation}\label{f-nn}
  f_{i}(x,w;u) = \sigma\bigg(\sum\limits_{j=1}^{n}w_{i,j}x_j +  u_i\bigg).
\end{equation}
That is, each node $i$ will sense the states of its neighbors (a neighbor being defined as any node $j$ where $w_{i,j} \neq 0$), compute a weighted sum of their values, add the external input $u_i$, and then apply the nonlinear function $\sigma$. Here, the nonlinear function is the logistic function $\sigma(x) = (1+\exp(-x))^{-1}$.

Iterating an attractor, to compute $f(x,w;u), f^{2}(x,w;u), \hdots,$ defines a dynamical system, and under certain conditions on the parameters of the network, this dynamical system converges to a fixed-point that is independent of the initial state of the system. This is formalized in the following proposition. This proposition uses the notion of an absolute norm  \cite{matanal}; this is any norm on Euclidean space such that
$\|(x_1,\hdots,x_n)\| = \|(|x_1|,|x_2|,\hdots,|x_n|)\|$. 
\begin{prop}
  Let $\|\cdot\|$ be any absolute norm.
  Then \eqref{f-nn} defines a contraction when
  $\|w\| < 4$,
  and a contraction coefficient is
  $\beta_x = \|w\|/4$.
\end{prop}
\begin{proof}
  Let $D$ be the diagonal matrix $D_{i,i} =  \sigma'\Big(\sum\limits_{j=1}^{n}w_{i,j}x_{j} + u_i\Big)$.
  Then the derivative of $f$ with respect to $x$ is
  $$  \frac{\partial f_i}{\partial x_j}(x,w;u) =
  D_{i,i}w_{i,j}
  $$
  and   $\frac{\partial f}{\partial x}(x,w;u) = Dw$.
  For any absolute norm $\|\cdot\|$ and diagonal matrix $D$,  it holds that
  $\|D\|= \max_{1\leq i\leq n}|D_{i,i}|$ (Theorem 5.6.3 in \cite{matanal}).
  Therefore $\left\|\frac{\partial f}{\partial x}(x,w;u)\right\| \leq \|D\|\|w\| \leq \|\sigma'\|_{\infty}\|w\|.$ Finally, note that $\|\sigma'\|_{\infty} = 1/4$.
\end{proof}

Note that norms that are absolute  include the $p$-norms $\|\cdot\|_1,\|\cdot\|_{2}$ and $\|\cdot\|_{\infty}$. 
\subsubsection{Problem formulation} We use the same notation as in Section \ref{prob-form} to define our dynamical system and objective function: For a set $u^1,\hdots,u^m$ of external inputs, we let $\widetilde{x}^1,\hdots, \widetilde{x}^{m}$ be approximate fixed-points for those inputs.
Then, define
$f^{\operatorname{Total}}$ and $e^{\operatorname{Total}}$ as :
\begin{align}
f^{\operatorname{Total}}((x^1,\hdots,x^{10}),w)
&=
\left(
f(x^{1},w;u^1),\hdots, f(x^{10},w;u^{10})
\right), \label{f-nn-total}
\\
e^{\operatorname{Total}}(x^{1},\hdots,x^{10})
&= \frac{1}{m}\sum\limits_{i=1}^{m}\|x^i - \widetilde{x}^i\|^{2}_{2}. \label{e-nn-total}
\end{align}
The optimization problem can then be formulated as
\begin{equation}\label{nn-problem}
  \begin{split}
  \min\limits_{w \in \reals^{n\times n}} &e^{\total}(x^1,\hdots,x^{10})\\
  \text{ subject to }
  &(x^1,\hdots,x^{10}) = f^{\total}((x^1,\hdots,x^{10}),w).
  \end{split}
\end{equation}
The function $f^{\total}$ is a parallel combination of systems that are contractions in  the norm $\|\cdot\|$. Therefore, $f^{\total}$ is itself a contraction in the norm $\sum\limits_{1\leq i \leq m}\|x_i\|$, with contraction coefficient
\begin{equation}\label{nn-contraction-coeff}
  \beta_x^{\total} = \|\sigma'\|_{\infty}\|w\|
  \end{equation}
  In this case, contraction is only guaranteed when $\|w\| < 4$ (in some norm $\|\cdot\|$). This means that optimization could become unstable or fail once $w$ grows large in magnitude, since there may no longer be a unique fixed point for the network. For the problems we considered, optimization always ended before this became an issue. We plot the observed contraction coefficient for one of the sample runs below in Figure \ref{fig:nn-epersist}.

\begin{figure}[t]
\centering
  \begin{subfigure}{    \includegraphics[width=.5\linewidth]{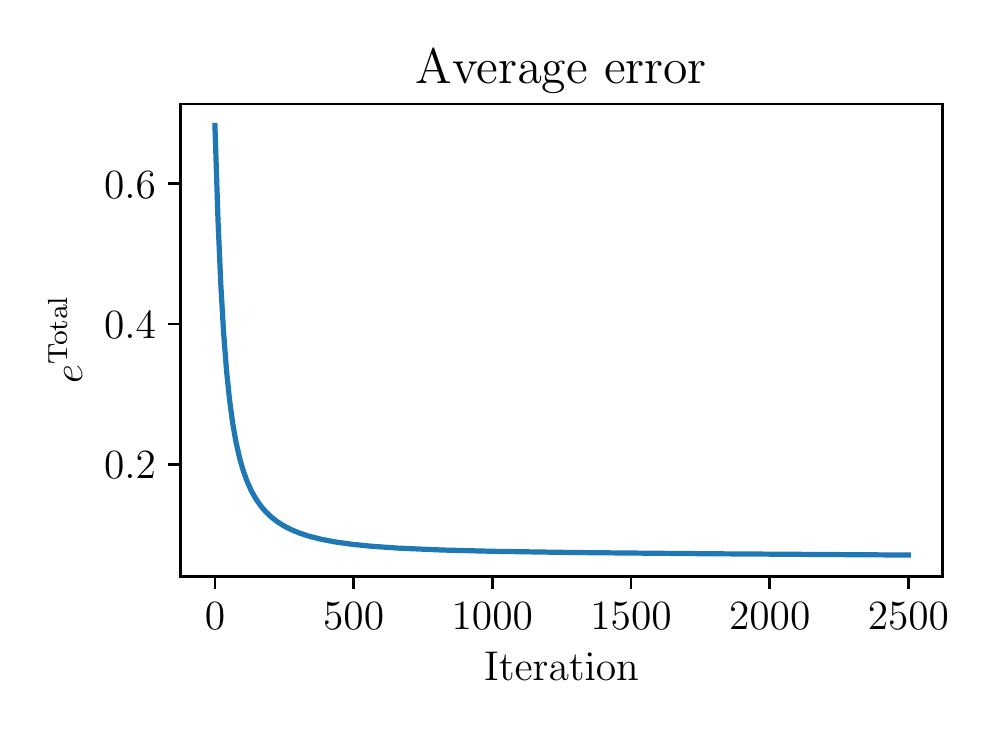}\vspace{-1em}}
  \end{subfigure}
  \begin{subfigure}{\hspace{-1em} \includegraphics[width=.5\linewidth]{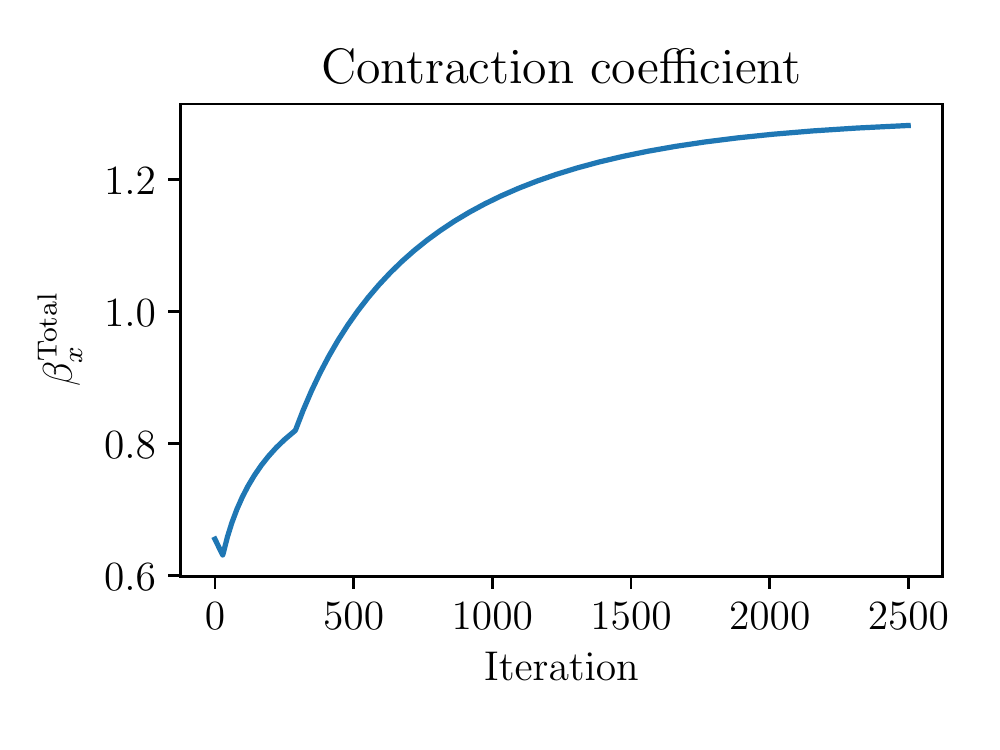}\vspace{-1em}}
  \end{subfigure}
\caption{Result for attractor network training. The left figure shows the value of the objective function as optimization progresses. The figure on the right  shows the contraction coefficient, given in Equation \eqref{nn-contraction-coeff}, at each iteration, relative to the norm $\|\cdot\|_{\infty}$\label{fig:nn-epersist}}
\end{figure}

\subsubsection{Optimization results}
We generated the training data using the following recipe. First, we generated a random matrix of weights by sampling the weights along each connection from $\N$. Then we generated $10$ input vectors $u^1,\hdots, u^{10}$, also by sampling entries from $\N$.  For each input vector, we approximately compute the fixed point of the network, obtaining the estimates $\widehat{x}^1,\hdots,\widehat{x}^{10}$.  Then, starting from a new random set of weights $w^0$, we train the network to map the inputs $u^i$ to the outputs $\widehat{x}^i$, using the persistent adjoint method.

Algorithm 1 was run with the following inputs. 
\begin{itemize}
\item
  $
  T = T^{Adj(f^{\total},e^{\total})},
  g = g^{Adj(f^{\total},e^{\total})}$  with $f,e$  as in \eqref{f-nn-total}, \eqref{e-nn-total}, resp.,
\item
  $ \epsilon = 0.4, \delta = 0.01 $
 \item
   $z_{0} = ((x^1,\hdots,x^{10}),(y^1,\hdots,y^{10}))$ is set to $(0,0) \in \mathbb{R}^{100}$
 \item
   $\|\cdot\|_W$ is the Frobenius norm $\|w\|_W = ( \sum_{i=1}^{n}\sum_{j=1}^{n}w_{i,j}^2)^{1/2}$. 
   \item 
$ \| \cdot \|_{Z} $ is $\|(x,y)\|_{Z} = \|x\|_{X} + \|x\|_{X^{*}}$ \quad (See \eqref{the-norm}.)
\end{itemize}

Figure \ref{fig:nn-epersist} shows the behavior of the attractor network as optimization progressive. 
The left plot in Figure \ref{fig:nn-epersist}  shows the error $E$ as training progresses, estimated by iterating $f^{\total}$ to a convergence threshold, and evaluating $e^{\total}(x)$ on the resulting estimate of the fixed-point. The right plot in Figure  \ref{fig:nn-epersist}  shows an estimate of the contraction coefficient (given in Equation \eqref{nn-contraction-coeff}.) In this case, we measure contraction relative to the norm $\|\cdot\|_{\infty}$ on the space of matrices. This means the norm $\|w\|$ is the maximum-absolute-row-sum. Note that the coefficient is greater than one at the end of training. However, we still observe that the error is nearly zero. This can be explained by the fact that condition $\|\sigma'\|_{\infty}\|w\|_{\infty} < 1$ is only a sufficient condition for contraction, and there may be another metric in which contraction could be verified, even when this fails to hold for the norm $\|\cdot\|_{\infty}$.

\section{Conclusion}
This article studied an algorithm for optimizing the fixed-point of a contraction mapping.
The algorithm is based on the construction of an auxiliary system out of the various derivatives of the underlying system $f$ and objective $e$.
This auxiliary process inherits the contraction property of the underlying system, 
and this allows us to apply a more general result about dynamic approximation to obtain gradient convergence for certain inputs to the algorithm. The procedure uses a dynamic time-scaling, in which the time spent computing derivative approximations is based on the magnitude of previous derivative estimates.

Our numerical results suggest the algorithm is practical. 
There are several extensions that may be of interest. These include the optimization of stochastic systems, where the problem is to optimize the stationary distribution of a Markov chain. 
In its current form, the convergence proof relies on Lipschitz constants to define the parameters $\epsilon, \delta$ and the norm $\|\cdot\|_{Z}$. In some problems of interest, these constants are difficult or impossible to bound, and when they are available the resulting constants may be too conservative. Therefore it would be useful to make the algorithm depend less on these quantities. 

\bibliography{pam}{}
\bibliographystyle{siam}

\end{document}